\documentclass[11pt,reqno]{amsart}
\usepackage{cite,epic,eepic,euscript,verbatim,amsmath,amssymb,amsthm,afterpage,float,bm,stmaryrd,color,ulem}
\usepackage{graphicx,epsfig,epstopdf}
\usepackage{appendix}
\usepackage{multirow}
\usepackage{comment}
\newtheorem{theorem}{Theorem}[section]
\newtheorem{lemma}{Lemma}[section]

\newtheorem{corollary}{Corollary}[section]

\newtheorem{remark}{Remark}[section]
\renewcommand{\theequation}{\thesection.\arabic{equation}}
\numberwithin{figure}{section}
\numberwithin{table}{section}

\def\XXint#1#2#3{{\setbox0=\hbox{$#1{#2#3}{\int}$}
\vcenter{\hbox{$#2#3$}}\kern-.51\wd0}}

\begin{document}
\title[A DG method for Poisson--Nernst--Planck systems]{A free energy satisfying discontinuous Galerkin method for one-dimensional Poisson--Nernst--Planck systems
}
\author[H.~Liu and Z.-M. Wang]{Hailiang Liu$^\dagger$ and Zhongming Wang$^\ddagger$\\  \\
 }
\address{$^\dagger$Iowa State University, Mathematics Department, Ames, IA 50011} \email{hliu@iastate.edu}
\address{$^\ddagger$ Florida International University,  Department of Mathematics and Statistics,  Miami, FL 33199}
\email{zwang6@fiu.edu}
\begin{abstract} We design an arbitrary-order free energy satisfying discontinuous Galerkin (DG) method
for solving time-dependent Poisson-Nernst-Planck systems. Both the semi-discrete and fully discrete DG methods are shown to satisfy the corresponding discrete free energy dissipation law for positive numerical solutions. Positivities of numerical solutions are enforced by an accuracy-preserving limiter in reference to positive cell averages. Numerical examples are presented to demonstrate the high resolution of the numerical algorithm and to illustrate  the proven properties of mass conservation,  free energy dissipation, as well as the preservation of steady states.
\end{abstract}

\maketitle

\section{Introduction}
In this paper,  we develop an arbitrary-order free energy satisfying numerical method
for solving the initial boundary value problem of the Poisson--Nernst--Planck (PNP) system,
 \begin{subequations}\label{PNP}
\begin{align}
  \partial_t c_i & = \nabla \cdot(\nabla c_i+ q_i c_i \nabla \psi)  \quad  x\in  \Omega, \; t>0 \\
- \Delta \psi & = \sum_{i=1}^m q_i c_i  +\rho_0(x),  \quad x\in  \Omega, \; t>0, \\
c_i(0, x) & =c_i^{\rm in}(x),  \quad x\in \Omega,\\
   \frac{\partial \psi}{\partial  \textbf{n}} & =\sigma, \quad \frac{\partial c_i}{\partial  \textbf{n}}+q_ic_i \frac{\partial \psi}{\partial  \textbf{n}} =0,\quad
 \quad x\in \partial\Omega,\; t>0,
\end{align}
\end{subequations}
where $c_i=c_i(t, x) $ is the local concentration of $i^{th}$ charged molecular or ion species with charge $q_i$ ($1\leq i\leq m$) at the spatial point $x$ and time $t$,   $\Omega \subset \mathbb{R}^d$  denotes a connected closed domain with smooth boundary $\partial\Omega$,  $\psi=\psi(t, x)$ is the electrostatic potential governed by the Poisson equation subject to the Neumann boundary data $\sigma$ and the charge density  that consists of both fixed charge $\rho_0$ and mobile ions,
the latter being a linear combination of all the concentrations $c_i$.  Here $\textbf{n}$ is the unit outward normal vector on the domain boundary $\partial \Omega$.  In this system,  the diffusion coefficient, the thermal energy and the dielectric coefficient have been normalized in a dimensionless manner.

The side conditions are necessarily compatible, i.e.,
\begin{equation}\label{compatibility}
\int_{\Omega} \left(\sum_{i=1}^m q_i c_i^{\rm in}(x) +\rho_0(x)\right)dx+\int_{\partial \Omega}\sigma ds=0,
 \end{equation}
 for the solvability of the problem.

The PNP system is a mean field approximation of diffusive molecules or ions, and  consists of Nernst-Planck (NP) equations that describe  the drift and diffusion of ion species, and the Poisson equation that describes the electrostatics interaction. In the process of charge transport the fluxes of charge carriers are driven exclusively by processes of diffusion and electric drift. This description using the flux traces back to Nernst \cite{Ne89} and Planck \cite{Pl80}, and is accurate for modeling systems with point charges.  Applications of this system are found in electrical engineering and electronkinetics \cite{ Mo83, Ma86, MRS90, Je96,  Da97, Li04},
  electrochemistry \cite{Gl42},  and biophysics \cite{Hi01, EL07}.

Three main  properties of the solution to (\ref{PNP}) are the non-negativity, mass conservation and the free energy dissipation, i.e.,
\begin{subequations}\label{str}
\begin{align}
& c_i^{\rm in}(x)  \geq 0  \Longrightarrow  c_i(t, x) \geq 0 \quad \forall t>0, \; x\in \Omega,\\
& \int_{\Omega} c_i(t,x)\,dx=\int_{\Omega} c_i^{\rm in}(x)\,dx \quad \forall t>0,\\
&  \frac{d}{dt}  F = -\sum_{i=1}^m \int_{\Omega} c_i^{-1}|\nabla c_i +q_ic_i \nabla \psi|^2 dx +\int_{\partial \Omega} \partial_t \sigma \psi ds,
\end{align}
\end{subequations}
where the  free energy $ F$ is defined by
\begin{align}\label{free}
F =  \int_{\Omega}  \sum_{i=1}^m  c_i {\rm log} c_i dx +\frac{1}{2}\int_{\Omega} |\nabla_x \psi|^2 dx.
\end{align}
The free energy contains both entropic part and  the interaction part: $c_i {\rm log} c_i$ is the entropy related to the Brownian motion of each ion species,  and $\frac{1}{2}\int_{\Omega} |\nabla_x \psi|^2 dx$ is the electrostatic potential of the Coulomb interaction between charged ions.  For the PNP system, 
the density is expected to converge to the equilibrium solution in a closed system  (e.g. $\sigma=0$)  regardless of how initial data are distributed.

These nice mathematical features are crucial for the analytical study of the PNP system.  For instance, 
by some energy estimate with the control of the free energy dissipation, the solution is shown to  converge to the thermal equilibrium state as time becomes large, if the boundary conditions are in thermal equilibrium (see, e.g.,  \cite{GG96}).  Long time behavior was studied in \cite{BHN94},  and further in \cite{AMT00, BD00} with refined convergence rates.  Results for the drift-diffusion model regarding existence and asymptotic studies in the case of different boundary conditions may be found in \cite{GG86, FI95a, FI95b}.

\subsection{Existing and proposed methods} Due to the wide variety of devices modeled by the PNP equations, computer simulation for this system of differential equations is of great interest. In addition to being more computationally efficient, PNP models more easily incorporate certain types of boundary conditions that arise in physical systems, such as boundaries of fixed concentration or electrostatic potential. However, the PNP equations present difficulties when computing approximate solutions: it is a strongly coupled system of  $m+ 1$ nonlinear equations, so that computational efficiency plays a critical role in applications of a numerical solver. The transport equations are often convection-dominated,  numerical simulation may produce negative ion concentrations or oscillations in the computed solution, if not  properly addressed.

 In the literature, there are different numerical approximations available for the steady state Poisson-Nernst-Planck equations, see e.g. \cite{ANRR05, BMMPSW05, Je96, Ma86, MRS90}.  Computational algorithms for time-dependent PNP systems have also been constructed for both one-dimensional and  two or three-dimensional models in various chemical/biological applications, and have been combined with the Brownian Dynamics simulations. The proposed algorithms range from finite difference methods \cite{CC65, KCN99,  CCK00, SL01, SLL03, BSDK09, ZCW11, MG14}  to finite element methods \cite{BZHBHM07, LHMZ10, CCAO14, PS09, MXL16}.  Many of existing algorithms are introduced to handle specific settings in complex applications, in which one may encounter different numerical obstacles, such as discontinuous coefficients, singular charges, geometric singularities, and nonlinear couplings to accommodate various phenomena  exhibited by biological ion channels. For a broader overview of proposed algorithms, we refer the interested reader to a recent review article \cite{WZCX12}.

Given the existing rich literature, the most distinct feature of this work is the use of properties (\ref{str}a)-(\ref{str}c) as a guide to
design an arbitrary high order algorithm  to efficiently simulate the solution at large times.  The most related works  to the present one are \cite{LW14, LW15, MXL16}  and in spirit \cite{GG96, CLP03, CF07,LY12, LY14a, LY14b}.
%
 The second order finite difference method introduced in \cite{LW14} satisfies all three properties in (\ref{str}) at the discrete level. For the class of nonlinear Fokker-Planck (NFP) equations
\begin{align}\label{fp+}
\partial_t c & =\nabla_x\cdot (f(c) \nabla_x (\psi(x)+H'(c)) ),
\end{align}
 with  the potential $\psi$ given, the high order discontinuous Galerkin method introduced in \cite{LW15} is shown to satisfy the discrete entropy dissipation law. If the potential $\psi$ is governed by the Poisson equation such as (\ref{PNP}b),  $f(c)=c$ and $H(c)=c\log c$, then \eqref{fp+} becomes the PNP system \eqref{PNP} with single species. A finite element method to the PNP system is recently introduced in \cite{MXL16} using a logarithmic transformation of the charge carrier densities, while the involved energy estimate resembles the physical energy law that governs the PNP system in the continuous case.   The main objective of this work is to develop a high order DG method that incorporates mathematical features (\ref{str})  to handle the coupling of the Poisson equation and the NP system so that the numerical solution remains faithful for long time simulations.


The discontinuous Galerkin (DG) method we present here is a class of finite element methods, using a completely discontinuous piecewise polynomial space for the numerical solution and the test functions. One main advantage of the DG method is the flexibility afforded by local approximation spaces combined with the suitable design of numerical fluxes crossing cell interfaces. More general information about DG methods for elliptic, parabolic, and hyperbolic PDEs can be found in the recent books and lecture notes see,e.g. \cite{ HW07,Ri08, Shu09}. The DG discretization we use here is motivated by  the direct discontinuous Galerkin (DDG) method proposed in \cite{LY09, LY10}. The main feature in the DDG schemes lies in numerical flux choices for the solution gradient, which involve higher order derivatives evaluated across cell interfaces. The entropy satisfying methods recently developed  in \cite{LW14, LW15, LY12, LY14a, LY14b}  are the main references for the present work. 

The main results in this paper include the formulation of an arbitrary-order DG methods, proofs of the discrete free energy dissipation and mass conservation for both semi-discrete and fully discrete DG methods, and an accuracy-preserving limiter to ensure the positivity preserving property for $c_i$. Numerical results are in excellent agreement with the analysis.

We point out that different boundary conditions  are used in applications, and they are important in driving the system out of equilibrium and produce the non-vanishing ionic fluxes. The numerical method presented in this work can be easily modified to incorporate different boundary conditions, through selection of appropriate boundary fluxes (see section 2).

\subsection{Related models}
The PNP system has many variants by simply altering the definition of the charge carrier flux. For example, the PNP system coupled with the Navier-Stokes equation is a basic model in the study of electronkinetics \cite{Li04}.  It  reduces to simpler models when  some of ion species become trivial or steady.  For example, in the case of $m=1$, if $c_1=c$ with $q_1=1$, we have
\begin{subequations}\label{PNP1}
\begin{align}
  \partial_t c & = \nabla \cdot(\nabla c+  c \nabla \psi)  \quad  x\in  \Omega,\; t>0, \\
- \Delta \psi & =  c  +\rho_0(x),  \quad x\in  \Omega, t>0, \\
c(0, x) & =c^{\rm in}(x),  \quad x\in \Omega,\\
   \frac{\partial \psi}{\partial  \textbf{n}} & =\sigma, \quad \frac{\partial c}{\partial  \textbf{n}}+c \frac{\partial \psi}{\partial  \textbf{n}} =0,\quad
  x\in \partial\Omega,\; t>0.
\end{align}
\end{subequations}
It may well be the case that when some species still evolve in time, the others are already at the steady states due to different time scales. For example, in the case of $m=2$, if $c_1=c$  and $c_2=S(x)$ with $q_1=1=-q_2$, we obtain
\begin{subequations}\label{PNP2}
\begin{align}
  \partial_t c & = \nabla \cdot(\nabla c +  c \nabla \psi)  \quad  x\in  \Omega, t>0,  \\
0 & = \nabla \cdot(\nabla S - S \nabla \psi)  \quad  x\in  \Omega, t>0,  \\
- \Delta \psi & =  c  +\rho_0(x) +S(x)  \quad x\in  \Omega, t>0, \\
c(0, x) & =c^{\rm in}(x),  \quad x\in \Omega,\\
   \frac{\partial \psi}{\partial  \textbf{n}} & =\sigma, \quad  \frac{\partial S}{\partial  \textbf{n}}- \sigma S=0,  \quad \frac{\partial c}{\partial  \textbf{n}}+ \sigma c  =0,
 \quad x\in \partial\Omega,\; t>0.
\end{align}
\end{subequations}
If all ion species are at the equilibrium states $c_i=\lambda_i e^{-q_i\psi}$, with $\lambda_i={\int_{\Omega} c^{\rm in}(x)dx}{\int_{\Omega} e^{-q_i\psi} dx}$
so that the total  density remains as given initially, the Poisson equation thus becomes a Poisson-Boltzmann equation(PBE) of the form
\begin{equation}\label{PBE}
-\Delta \psi=\sum_{i=1}^m \left(q_i\int_{\Omega}c^{\rm in}_i(x)dx\right)\frac{e^{-q_i\psi}}{\int_{\Omega} e^{-q_i\psi} dx} +\rho_0(x),  \quad x\in \Omega, \quad  \frac{\partial \psi}{\partial  \textbf{n}}\Big|_{\partial \Omega} =0.
\end{equation}
We should  point out that the numerical method presented in this paper may be used as an iterative algorithm to numerically compute the reduced system (\ref{PNP2}) and the nonlocal PBE (\ref{PBE}),  which plays essential roles in chemistry and biophysics.

In a larger context, the concentration equation also links to the general class of aggregation equations with diffusion
\begin{equation}\label{u}
\partial_t c +\nabla\cdot (c\nabla (G*c))=\Delta c,
\end{equation}
 which has been widely studied in applications such as biological swarms \cite{BCM07, BF08, TBL06} and  chemotaxis  \cite{BRB11, BCL09}.  For chemotaxis,  a wide literature exists in relation to the Keller-Segel model (see \cite{BRB11, BCL09} and references therein).  The left-hand-side in (\ref{u}) represents the active transport of the density $c$ associated to a non-local velocity field $u =\nabla (G*c)$. The potential $G$ is usually assumed to incorporate attractive interactions among individuals of the group, while repulsive (anti-crowding) interactions are accounted for by the  diffusion in the right-hand-side.  Of central role in studies of model (\ref{u}), and also particularly relevant to the present research, is the gradient flow formulation of the equation with respect to the free energy
\begin{equation}\label{eu}
F[c]=\int_{\Omega} c \log c dx -\frac{1}{2}\int_{\Omega}\int_{\Omega} G(x-y)c(x)c(y)dxdy.
\end{equation}

\subsection{Contents}  This paper is organized as follows: in Section 2, we present the DG method for the one dimensional case, and discuss how to deal with different types of boundary conditions; Section 3 is devoted to theoretical analysis for both semi-discrete and fully discrete schemes. We give the details of the numerical algorithm in Section 4, including how to compute the potential $\psi_h$, and the positivity preserving limiter.  Numerical results are presented in Section 5. Finally,  concluding remarks are given in Section 6.  Further numerical implementation details are given in the appendix.

\section{DG discretization in space}
\subsection{The DG scheme} In this section we present a  DG scheme for (\ref{PNP}). We consider a domain $\Omega = [a,  b]$ and a mesh, which
is not necessarily uniform i.e.,  a family of $N$ control cells  $I_j$ such that $I_j =(x_{j-1/2}, x_{j+1/2})$  with cell center
$x_j =(x_{j-1/2} +x_{j+1/2})/2$.
We set
$$
a=x_{1/2}<x_1<\cdots <x_{N-1/2}<x_N<x_{N+1/2}=b,
$$
and $\Delta x_j=x_{j+1/2}-x_{j-1/2}$.

Define the discontinuous finite element space
$$
V_h=\{v\in L^2(\Omega), \quad v|_{I_j}\in P^k(I_j), j=1, \cdots, N\},
$$
where $P^k$ denotes polynomials of degree at most $k$.   We rewrite the PNP system as follows
\begin{subequations}\label{cqp}
\begin{align}
 \partial_t c_i&=\partial_x (c_i \partial_x p_i),  \; i=1, \cdots, m, \\
 p_i &=q_i \psi +\log c_i,\\
 -\partial_{x}^2\psi&=\sum_{i=1}^m q_ic_i +\rho_0(x),
\end{align}
\end{subequations}
subject to initial data $c_i(0, x)=c_i^{\rm in}(x)$, which are necessarily  to meet the compatibility requirement:
$$
\int_{\Omega} \left(\sum_{i=1}^m q_i c_i^{\rm in}(x) +\rho_0(x)\right)dx=\int_{\partial \Omega}\sigma ds.
$$

The DG scheme is to find $c_{ih}, p_{ih}, \psi_h \in V_h$ such that for all $v_i, r_i, \eta \in V_h$, $i=1, \cdots, m$,
\begin{subequations}\label{dg}
\begin{align}
& \int_{I_j} \partial_t c_{ih} v_i dx =-\int_{I_j} c_{ih} \partial_x p_{ih} \partial_x v_i dx +\{c_{ih}\} \left( \widehat{\partial_x p_{ih}}v_i + (p_{ih}-\{p_{ih}\})
\partial_x v_{i}\right)\Big|_{\partial I_j},\\
 & \int_{I_j}p_{ih} r_i dx =\int_{I_j}(q_i \psi_h + \log c_{ih} )r_{i}dx,\\
 & \int_{I_j}  \partial_x \psi_{h}\partial_x \eta  dx - \left( \widehat { \partial_x\psi_{h}} \eta   +(\psi_h - \{\psi_h\})\partial_x \eta \right)\Big|_{\partial I_j}  = \int_{I_j}
  \left(\sum_{i=1}^m q_i c_{ih} +\rho_0\right)
 \eta dx,
\end{align}
\end{subequations}
where  $\widehat{\partial_x p_{ih}}=Fl(p_{ih})$ and $\widehat{ \partial_x\psi_{h}} =Fl(\psi_h)$, with the flux operator $Fl$ defined by
\begin{align}\label{flux}
Fl(w):=\beta_0 \frac{[w]}{h} +\{\partial_xw\}+\beta_1h[\partial_x^2 w].
\end{align}
Here we have used the notation: $[q]=q^+-q^-$, and $\{q\}={(q^++q^-)}/{2}$, where $q^+$ and $q^-$ are the values of $q$ from the right and left cell interfaces, respectively. The parameters $\beta_0$ and $\beta_1$ in $Fl(w)$ are chosen to enforce the free energy satisfying property. The details of how to choose $\beta_0$ and $\beta_1$ are given and justified
in Section \ref{secDGproperty}.

For the ODE system \eqref{dg}, we prepare initial data $c_{ih}^{{\rm in}}$ by the projection of $c_{i}^{\rm in}$ on the discontinuous finite element  space $V_h$ so that
\begin{equation}\label{proj}
\int_{I_j} c_{ih}^{\rm in}(x) v(x)dx=\int_{I_j} c_{i}^{\rm in}(x)v(x)dx, \quad \forall v\in P^k(I_j), \quad i=1, \cdots, m.
\end{equation}

\subsection{Boundary conditions} \label{BC}   The boundary conditions are a critical component of the PNP model and determine important qualitative behavior of the solution. {
Various boundary conditions may be used with the PNP equations. Here we consider the simplest form of boundary conditions for the concentration and for the electric potential field. For the concentration, one usually ignores any chemical reactions at the solid boundaries and assumes a no flux condition \cite{PS09, MG14, MXL16}, i.e.,
$$
 \frac{\partial c_i}{\partial  \textbf{n}}+q_ic_i \frac{\partial \psi}{\partial  \textbf{n}} =0,\quad
 \quad x\in \partial\Omega,\; t>0.
$$
However, it should be noted that it is possible to account for reactions through, for example, using the Butler--Volmer reaction kinetics \cite{BCB05}. The boundary conditions for the electrostatic potential are however not unique and greatly depend on the problem under investigation. Usually, one ends up either specifying the potential or the charge density on the boundary, which corresponds to Dirichlet or Neumann boundary conditions, respectively; or using }
Robin boundary conditions which model capacitors at the boundary.  Any combination of these boundary conditions can be applied to $\psi$. Therefore,  at the domain boundary the numerical flux needs to be treated with care so that the pre-specified boundary conditions are properly enforced.

To incorporate the boundary condition (\ref{PNP}d), we set  at  both $x_{1/2}$ and $x_{N+1/2}$,
 \begin{align} \label{bp}
&  Fl(p_{ih})=0, \; \{p_{ih}\}=p_{ih}, \; \{c_{ih}\}=c_{ih}\quad i=1, \cdots, m,  \\ \label{bpsi}
& Fl(\psi_{h})=\sigma, \; \{\psi_{h}\}=\psi_{h}.
 \end{align}
 For other types of boundary conditions,  we only need to modify the boundary fluxes accordingly.  For instance, for Dirichlet boundary conditions of the form
 $$ \psi(t,a)=\psi_l, \quad \psi(t,b)=\psi_r; \quad c_i(t,a)=c_{il},   \quad c_i(t,b)=c_{ir},
 $$
 we define  $\{c_{ih}\}$, $\{\psi_{ih}\}$, $Fl(\psi_{ih})$  and $Fl(p_{ih})$ at the boundary in the following way,
  \begin{align}
   at \quad &x_{1/2}  \label{Dirichleta} \\
  &\{c_{ih}\}= \frac{1}{2}(c_{ih}^++c_{il}),\quad\{\psi_h\}=\frac{1}{2}(\psi_l +\psi_h^+), \notag \\
  &Fl(\psi_{hx})=\beta_0(\psi_h^+-\psi_l)/h +\psi_{hx}^+ , \notag\\
  &{Fl(p_{ih})=-\beta_0(q_i\psi_l+\log c_{il} - p_{ih}^+)/h+p_{ihx}^+}, \notag\\
at \quad &x_{N+1/2} \label{Dirichletb}\\
& \{c_{ih}\}= \frac{1}{2}(c_{ir}+c_{ih}^-),\quad   \{\psi_h\}=\frac{1}{2}(\psi_r +\psi_h^-), \notag  \\
&Fl(\psi_{hx})=\beta_0(\psi_r - \psi_h^-)/h +\psi_{hx}^- ,\notag\\
  &{Fl(p_{ih})=\beta_0(q_i\psi_{r}+\log c_{ir} - p_{ih}^-)/h+  p_{ihx}^-,} \notag
\end{align}
where $\beta_0$ is the penalty parameter to be chosen.
{
\subsection{Comments on the generalization for multidimensional case} It is straightforward to extend DG formulation (\ref{dg}) to multi-dimensional spaces. Let $\Omega$ be a convex, bounded polygonal domain in $\mathbb{R}^d (d=2, 3)$.  We partition $\Omega$ into computational elements denoted by $\mathcal{T}_h = \{K\}$, and $h$  denotes the characteristic length of all the elements of $\mathcal{T}_h$.   As usual, we assume the mesh is conforming and the subdivision is regular.
We set the DG finite element space as
$$
V_h=\{v\in L^2(\Omega): \forall K\in \mathcal{T}_h, \; v|_K \in P^k(K)\},
$$
where $P^k(K)$ is the space of polynomial functions of degree at most $k$  on $K$.   With such an approximation space, the semi-discrete DG scheme  is to find $c_{ih}, p_{ih}, \psi_h \in V_h$ such that for all $v_i, r_i, \eta \in V_h$, $i=1, \cdots, m$,
\begin{subequations}\label{dg+}
\begin{align}
& \int_{K} \partial_t c_{ih} v_i dx =-\int_{K} c_{ih} \nabla_x  p_{ih}\cdot \nabla_x v_i dx +\int_{\partial K} \{c_{ih}\} \left( \widehat{\partial_n p_{ih}}v_i + (p_{ih}-\{p_{ih}\})
\partial_n v_{i}\right)ds,\\
 & \int_{K}p_{ih} r_i dx =\int_{K}(q_i \psi_h + \log c_{ih} )r_{i}dx,\\
 & \int_{K}  \nabla_x \psi_{h} \cdot \nabla_x \eta  dx - \int_{\partial K} \left( \widehat { \partial_n \psi_{h}} \eta   +(\psi_h - \{\psi_h\})\partial_n \eta \right)ds  = \int_{K}
  \left(\sum_{i=1}^m q_i c_{ih} +\rho_0\right)
 \eta dx,
\end{align}
\end{subequations}
where  $\widehat{\partial_n p_{ih}}=Fl_n(p_{ih})$ and $\widehat{ \partial_n \psi_{h}} =Fl_n(\psi_h)$, with the flux operator $Fl_n$ defined on the interface $e$ by
\begin{align}\label{flux}
Fl_n(w):=\beta_0 \frac{[w]}{h_e} +\{\partial_n w\}+\beta_1h_e[\partial_n^2 w].
\end{align}
Here the normal vector $n$ is assumed to be oriented from $K_1$ to $K_2$, sharing a common edge (face) $e$,
and $h_e$ is the characteristic length of $e$.  $\partial_n w$ and $\partial_n^2 w$ denote the first and second order derivative along direction $n$, respectively.  The average $\{w\}$ and the jump $[w]$ of $w$  on $e$  are as follows:
$$
\{w\}=\frac{1}{2} (w|_{K_1}+w|_{K_2}), \quad [w]= w|_{K_2} - w|_{K_1} \quad \forall e\in \partial K_1\cap \partial K_2.
$$
For $e$ in the set of boundary edges, each numerical solution has a uniquely defined restriction on $e$, while boundary conditions can be weakly enforced through the boundary fluxes in the same way as in the one-dimensional case (see Section 2.2). Most of the one-dimensional analysis presented in this work can be easily carried over to the multi-dimensional case, as long as both the bound $\Gamma(\beta_1, 1)$ as in (\ref{gb}), and  some mesh-dependent inverse inequalities such as those in Lemma 3.2 can be established.  However, for this coupled nonlinear system the implementation of the multi-dimensional algorithm is much more involved, and will be left in future work.
}

\section{Properties of the PNP system and the DG method } \label{secDGproperty}

\subsection{Properties of the continuous PNP system}
Under the zero flux condition for $c_i$ in (\ref{PNP}d) and the evolution law (\ref{PNP}a), the total mass for each $c_i (i=1, \cdots, m) $ is conserved in the sense that
\begin{align*}
\frac{d}{dt}\int_{\Omega}c_i(t, x)dx & =\int_{\Omega} \nabla_x \cdot( \nabla c_i +q_ic_i \nabla \psi)dx \\
& =\int_{\partial \Omega} \left(\frac{\partial c_i}{\partial  \textbf{n}}+q_i c_i \frac{\partial \psi}{\partial  \textbf{n}} \right) ds \\
&=0.
\end{align*}
Furthermore, assume the Dirichlet boundary condition is given and homogeneous, then the stability of the solution to the PNP system
is known \cite{BHN94, BD00} to be given by the energy law of the form
$$
\frac{d}{dt} \left\{  \int_{\Omega} \sum_{i=1}^m c_i({\rm log} c_i -1) +\frac{1}{2}|\nabla_x \psi|^2  dx \right\}
=-\int_{\Omega} \sum_{i=1}^m c_i|\nabla_x ({\rm log} c_i +q_i\psi)|^2dx,
$$
where the functional,
$$
\int_{\Omega} \sum_{i=1}^m c_i({\rm log} c_i -1) +\frac{1}{2}|\nabla_x \psi|^2dx
$$
is the energy and
$$
\int_{\Omega} \sum_{i=1}^m c_i|\nabla_x ({\rm log} c_i +q_i\psi)|^2dx \geq 0
$$
is the rate of dissipation.  For the Neumann boundary conditions, we show the corresponding energy law by a formal calculation.
First, the  free energy (\ref{free}) by the Poisson equation can be rewritten as
\begin{align*}
F & = F_1+F_2: =\sum_{i=1}^m  \int_{\Omega}  c_i {\rm log} c_i dx  +\frac{1}{2} \int_{\Omega}|\nabla_x \psi|^2 dx.
\end{align*}
We  formally calculate the change rate of each part of the free energy
\begin{align*}
\frac{d}{dt}F_1(t) & = \sum_{i=1}^m \int_{\Omega} (1+{\rm log} c_i)\partial_t c_idx \\
& =- \sum_{i=1}^m \int_{\Omega} \nabla_x ({\rm log} c_i)\cdot (\nabla_x c_i+q_ic_i \nabla_x \psi)dx\\
& = - \sum_{i=1}^m \int_{\Omega} \left( c_i^{-1}|\nabla_x c_i|^2  +q_i\nabla_x c_i \cdot \nabla_x \psi \right)dx
\end{align*}
and
\begin{align*}
\frac{d}{dt}F_2(t) & =  \int_{\Omega} \nabla_x \psi \cdot \partial_t (\nabla_x \psi) dx \\
& = \int_{\partial \Omega} \psi \partial_t (\partial_n\psi) ds +  \int_{\Omega}  \psi \partial_t ( -\Delta \psi) dx\\
&= \int_{\partial \Omega} \psi \partial_t \sigma ds +\sum_{i=1}^m \int_{\Omega} \psi  ( q_i \partial_t c_i)dx\\
&= \int_{\partial \Omega} \psi \partial_t \sigma ds -\sum_{i=1}^m \int_{\Omega} \nabla_x  (q_i \psi)  \cdot  (\nabla_x c_i +q_ic_i\nabla_x \psi) dx\\
& =\int_{\partial \Omega} \psi \partial_t \sigma ds - \sum_{i=1}^m \int_{\Omega} q_i \nabla_x  \psi  \cdot  \nabla_x c_i dx -
\sum_{i=1}^m \int_{\Omega} q_i^2 c_i |\nabla_x \psi|^2 dx.
\end{align*}
Putting together we obtain
$$
\frac{d}{dt}F=\int_{\partial \Omega} \psi \partial_t \sigma ds -\sum_{i=1}^m \int_{\Omega} c_i^{-1} |\nabla_x c_i +q_ic_i\nabla_x \psi|^2dx.
$$
For the case of $\partial_t \sigma=0$ we have
$$
\frac{d}{dt}F=  -\sum_{i=1}^m \int_{\Omega} c_i^{-1} |\nabla_x c_i +q_ic_i\nabla_x \psi|^2dx \leq 0.
$$
\begin{remark} In the special case that  $\sigma=0$ and $\rho_0=0$, both $F_1$ and $F_2$ are each decreasing in time, with
\begin{align*}
\frac{d}{dt}F_1 & =  - \int_{\Omega}  \sum_{i=1}^m c_i^{-1}|\nabla_x c_i|^2dx  -\int_{\Omega}|\Delta \psi|^2dx \leq 0, \\
\frac{d}{dt} F_2 & =- \int_{\Omega} \left( \sum_{i=1}^m q_i^2 c_i\right)  |\nabla_x \psi|^2 dx-\int_{\Omega}|\Delta \psi|^2dx \leq 0.
\end{align*}
This can be verified by a direct calculation as
\begin{align*}
 - \sum_{i=1}^m \int_{\Omega} q_i \nabla_x  \psi  \cdot  \nabla_x c_i dx  & =\int_{\Omega} (\Delta \psi) \sum_{i=1}^m (q_ic_i)dx -
 \int_{\partial \Omega} \sigma \sum_{i=1}^m (q_ic_i)ds \\
 & =-\int_{\Omega}|\Delta \psi|^2dx.
\end{align*}
\end{remark}

\subsection{Properties of the semi-discrete DG method}
For some $M(x)>0$, we define the weighted  bilinear operator
$$
 A_M(u, v)=  \sum_{j=1}^N \int_{I_j} M \partial_x u \partial_x v dx + \sum_{j=1}^{N-1} \left( \{M\} (\widehat {u_x}[v]  +\{v_x\}[u] )\right)_{j+1/2},
$$
and the associated energy norm
$$
\|v\|_{M, E}^2= \sum_{j=1}^N \int_{I_j}M v_x^2 dx +\sum_{j=1}^{N-1}\frac{\beta_0}{h} \left(\{M\}[v]^2 \right)_{j+1/2}.
$$
When $M=1$, we denote $\|v\|_{1, E}=\|v\|_E$. We also introduce the notation,
\begin{equation}\label{b1m}
\Gamma(\beta_1, M): =\sup_{u\in P^{(k-1)}([-1, 1])}\frac{2(u(1)-2\beta_1\partial_\xi u(1))^2}{\int_{-1}^1 w_j(\xi)u^2(\xi)d\xi}
\end{equation}
with
$$
w_j(\xi)=\min\left\{M(x_j+\frac{h}{2}\xi), M(x_{j+1}-\frac{h}{2}\xi) \right\}.
$$
We recall the following estimates.
\begin{lemma}\label{lem1}(cf. \cite{LY14a})
If $\beta_0>\Gamma(\beta_1, M)\{M\}$ at each interface $x_{j+1/2}$, $j=1, \cdots, N-1$, then
\begin{equation}\label{amv}
A_M(v, v) \geq \gamma \|v\|_{M, E}^2 +\frac{1}{2}\int_{I_1\cup I_N} M v_x^2dx
\end{equation}
for some $\gamma\in (0, 1)$. We also have
\begin{equation}\label{g1m}
\Gamma(\beta_1, M)  \geq \frac{2h([\partial_x w] +\beta_1h[\partial_x^2])^2}{\left(\int_{I_j}+\int_{I_{j+1}}\right)
M (\partial_x w)^2dx}.
\end{equation}
\end{lemma}
\begin{remark} The additional terms $\frac{1}{2}\int_{I_1\cup I_N} M v_x^2dx$ were not used in \cite{LY14a}; here we need them to control boundary contributions.
\end{remark}
In the case of $M=1$, the bound $\Gamma(\beta_0, 1)$ is obtained in \cite{Liu14}:
\begin{equation}\label{gb}
\Gamma(\beta_1, 1)= k^2 \left(1-\beta_1 (k^2-1)+\frac{\beta_1^2}{3}(k^2-1)^2 \right),
\end{equation}
which achieves its minimum $k^2/4$ at $\beta_1=\frac{3}{2(k^2-1)}$.

We first examine the existence and uniqueness of $\psi_h$ from solving (\ref{dg}c).  Note that for Neumann boundary data, only $\sigma_a$ and $\sigma_b$ are pre-specified so that
$$
\sigma_a-\sigma_b=\int_{a}^b \rho(t, x)dx, \quad  \rho(t, x)=\rho_0(x)+\sum_{i=1}^m q_ic_{i}(t, x),
$$
yet $\psi(t, a)$ and $\psi(t, b)$ are not given, the exact solution is unique up to an additive constant.

\begin{theorem} Assume that $\beta_0>\Gamma(\beta_1, 1)$, then for each given  $\rho$,  there exists a unique  $\psi_h$ to (\ref{dg}c) with (\ref{flux}) and (\ref{bpsi}), up to an additive constant.
\end{theorem}
\begin{proof}
It suffices to prove the uniqueness since existence is equivalent to uniqueness for this finite dimensional problem.  Let $\eta =\psi_{h1}-\psi_{h2}$ for two
different solutions $\psi_{hi}$, then
$$
A_1(\eta, \eta)=0.
$$
On the other hand, Lemma 3.1 ensures that there exists $\gamma \in (0, 1)$ such that
$$
A_1(\eta, \eta) \geq  \gamma\|\eta\|_E^2= \gamma \left(\sum_{j=1}^N \int_{I_j}\eta_x^2dx +\sum_{j=1}^{N-1}\frac{\beta_0}{h}[\eta]_{j+1/2}^2 \right).
$$
Hence $\eta$ must be a constant.
\end{proof}

We next prove the conservation of $c_{ih}$ and the semi-discrete free energy dissipation law.
\begin{theorem}\label{thm3.3}
\begin{itemize}
\item[1.]  The semi-discrete scheme is  conservative in the sense that  each  total concentration $c_{ih} (i=1, \cdots,  m)$ remains unchanged in time,
\begin{align}
\label{conserve1}
\frac{d}{dt} \sum_{j=1}^N \int_{I_j} c_{ih}(t, x) dx  =0, \quad t>0.
\end{align}
\item[2.] Suppose that $\partial_t \sigma=0$ and $c_{ih}(t,x)>0 $ in each $I_j$, the semi-discrete free energy
$$
F =\sum_{j=1}^N  \int_{I_j}  \left[ \sum_{i=1}^m c_{ih} {\rm log} c_{ih}+\frac{1}{2} \left( \sum_{i=1}^m q_i c_{ih}+\rho_0 \right) \psi_h \right] dx +\frac{1}{2}\int_{\partial \Omega}\sigma  \psi_h ds
$$
satisfies
 \begin{align}
\frac{d}{dt} F&= - \sum_{i=1}^m A_{c_{ih}}(p_{ih}, p_{ih}).
\label{fd}
\end{align}
Moreover,
$$
\frac{d}{dt} F\leq 0,
$$
provided $\beta_0$ is suitably large, and $\beta_1=0$ in $Fl(\psi_h)$  defined in (\ref{flux}) and (\ref{bpsi}).
\end{itemize}
\end{theorem}
\begin{proof}
1.  The conservation \eqref{conserve1} follows from taking  $v_i=1$ in (\ref{dg}a) with summation over $j$.

2. Set
$$
\rho_h(t, x)=\sum_{i=1}^{m} q_i  c_{ih}(t, x) +\rho_0(x).
$$
A direct calculation using $\sum_{j=1}^N \int_{I_j} \partial_tc_{ih}dx =0$ and the assumption of $\sigma_t=0$ gives
\begin{align*}
\frac{d}{dt}F&=  \sum_{j=1}^N \int_{I_j} \sum_{i=1}^{m}(1+\ln c_{ih} +q_i\psi_h  )\partial_tc_{ih} dx    +\frac{1}{2}
 \sum_{j=1}^N \int_{I_j}  [\rho_h \partial_t \psi_h - \partial_t \rho_{h} \psi_h ]dx + \frac{1}{2}\int_{\partial \Omega}\sigma \partial_t \psi_h ds\\
    &  =  \sum_{j=1}^N \int_{I_j} \sum_{i=1}^{m} (\ln c_{ih}+q_i\psi_h)\partial_t{c_{ih}} dx  +\frac{1}{2}
 \sum_{j=1}^N \int_{I_j}  [\rho_h \partial_t \psi_h - \partial_t \rho_{h} \psi_h ]dx + \frac{1}{2}\int_{\partial \Omega}\sigma \partial_t \psi_h ds.
\end{align*}
From (\ref{dg}b) with $r_i=\partial_t c_{ih}$ it follows that
$$
\int_{I_j} (\ln c_{ih}+q_i\psi_h)\partial_t {c_{ih}} dx =\int_{I_j} p_{ih}\partial_t c_{ih}dx.
$$
 Summing (\ref{dg}a) in $i, j$ with $v_i=p_{ih}$ gives
\begin{align*}
&\sum_{j=1}^N \int_{I_j} \sum_{i=1}^m \partial_t c_{ih} p_{ih} dx = - \sum_{j=1}^N \int_{I_j}  \sum_{i=1}^m c_{ih} |p_{ihx}|^2 dx - \sum_{j=1}^{N-1} \sum_{i=1}^m ( \{c_{ih}\} [p_{ih}] (\widehat{p_{ihx}} +\{p_{ihx}\} ))_{j+1/2},
\end{align*}
since $\widehat{p}_{hx, \frac{1}{2}}=\widehat{p}_{hx,N+\frac{1}{2}}=0$.   Next we sum  (\ref{dg}c) over all cell index $j$ to obtain
\begin{align}\label{pp}
A_1(\psi_h, \eta) =\sum_{j=1}^N\int_{I_j} \rho_h \eta dx+ \int_{\partial \Omega} \sigma \eta ds.
\end{align}
In one dimensional case,
$$
\int_{\partial \Omega} \sigma \eta ds =\sigma_b \eta_{N+1/2}^- - \sigma_a \eta_{1/2}^+.
$$
Taking derivative of (\ref{pp}) with respect to $t$ we also have
\begin{align}\label{pptsd}
A_1(\partial_t \psi_h, \eta) =\sum_{j=1}^N\int_{I_j} \partial_t \rho_h \eta dx.
\end{align}
Thus (\ref{pp}) with $\eta=\partial_t\psi_h$ subtracted  by (\ref{pptsd}) with $\eta=\psi_h$ gives
\begin{align*}
  \sum_{j=1}^N \int_{I_j}  [\rho_h \partial_t \psi_h - \partial_t \rho_{h} \psi_h ]dx + \int_{\partial \Omega}\sigma \partial_t \psi_h ds
 & =A_1(\psi_h, \partial_t \psi_h )- A_1(\partial_t \psi_h, \psi_h) \\
 & =\beta_1 h \sum_{j=1}^{N-1} ([\partial_x^2 \psi_h][ \partial_t \psi_h] -[\partial_x^2\partial_t \psi_h][\psi_h])_{j+1/2}\\
 &=0
\end{align*}
since $\beta_1=0$ in the numerical flux $Fl(\psi_h)$ defined in (\ref{flux}).

Putting all together we obtain
\begin{align*}
\frac{d}{dt}{F} & =- \sum_{j=1}^N \int_{I_j}  \sum_{i=1}^m c_{ih} |p_{ihx}|^2 dx - \sum_{j=1}^{N-1} \sum_{i=1}^m ( \{c_{ih}\} [p_{ih}] (\widehat {p_{ihx}} +\{p_{ihx}\} ))_{j+1/2} \\
& =- \sum_{i=1}^m A_{c_{ih}}(p_{ih}, p_{ih}) \\
& \leq 0,
\end{align*}
if at each interface $\beta_0> \max_i \Gamma(\beta_1, c_{ih})\{c_{ih}\}$, as shown in Lemma \ref{lem1}.

\begin{remark} For $(\beta_0, \beta_1)$ in the numerical flux formulation (\ref{flux}),  we use $\beta_0>k^2$ and $\beta_1=0$ for $Fl(\psi_{h})$.  For $Fl(p_{h})$ we need to pick
$$
\beta_0>\Gamma(\beta_1,c_{ih})\{c_{ih}\}
$$
for numerical solutions  $c_{ih}$.  For sufficiently small $h$, variation of $c_{ih}$ near each interface is relatively bounded by a factor  $2$,  hence it suffices to  choose
$$
\beta_0> 2 \Gamma(\beta_1, 1) =2 k^2 \left(1-\beta_1 (k^2-1)+\frac{\beta_1^2}{3}(k^2-1)^2 \right),
$$
where (\ref{gb}) has been used.
For $k>1$, one choice  is  $\beta_1=\frac{3}{2(k^2-1)}$ so that $\beta_0> 2\Gamma(\beta_1, 1)=k^2/2$.  Therefore it suffices to take  $\beta_0 >k^2$ in (\ref{flux}) for both $Fl(\psi_{h})$ and $Fl(p_{h})$.  The choices  of $(\beta_0, \beta_1)$ in our numerical simulation all fall within the range $\beta_0>k^2$.     
\end{remark}

\end{proof}
\subsection{Properties of the fully discrete DG method } In order to preserve the {free energy} dissipation law at
each time step, the time step restriction is needed when using an explicit time discretization. We
now discuss this issue by taking the Euler first order time discretization of \eqref{dg} with  uniform time step $\Delta t$:
find $c_{ih}^{n+1}, \psi^{n+1}_h \in V_h$ such that for any $v_i, r_i, \eta \in V_h$,
\begin{subequations}\label{dgEuler}
\begin{align}
& \int_{I_j} D_t c_{ih}^n v_i dx =-\int_{I_j} c_{ih}^n \partial_x p_{ih}^n \partial_x v_i dx +\{c_{ih}^n\} \left( \widehat{\partial_x p_{ih}^n}v_i + (p_{ih}^n-\{p_{ih}^n\})
\partial_x v_{i}\right)\Big|_{\partial I_j},\\
 & \int_{I_j}p_{ih}^n r_i dx =\int_{I_j}(q_i \psi_h^n + \log c_{ih}^n )r_{i}dx,\\
 & \int_{I_j}  \partial_x \psi_{h}^n \partial_x \eta  dx - \left( \widehat { \partial_x \psi_{h}^n} \eta   +(\psi_h^n - \{\psi_h^n\})\partial_x \eta \right)\Big|_{\partial I_j}  = \int_{I_j}
  \left[\sum_{i=1}^m q_i c_{ih}^n +\rho_0\right]
 \eta dx.
\end{align}
\end{subequations}
Here and in what follows, we use the notation for any function $w^n(x)$ as
$$
D_t w^n=\frac{w^{n+1}-w^n}{\Delta t}.
$$
We recall some basic estimates:
\begin{lemma} (cf. \cite[Lemma 3.2]{LW15}) The following  inequalities hold for any $v\in V_h$:
\begin{subequations}\label{in}
\begin{align}
& \sum_{j=1}^N \int_{I_j} v_x^2 dx \leq \frac{k(k+1)^2(k+2)}{h^2} \sum_{j=1}^N  \int_{I_j}v^2 dx, \\
& \sum_{j=1}^{N-1} [v]^2_{j+1/2}  \leq \frac{4(k+1)^2}{h} \sum_{j=1}^N  \int_{I_j}v^2 dx, \\
& \sum_{j=1}^{N-1} \{v_x\}^2_{j+1/2} \leq \frac{k^3(k+1)^2(k+2)}{h^2} \sum_{j=1}^N  \int_{I_j}v^2 dx.
\end{align}
\end{subequations}
\end{lemma}
Based on these estimates we show several properties of the fully discretized scheme.
\begin{theorem}\label{thmEuler}
\begin{itemize}
\item[1.]  The fully discrete scheme (\ref{dgEuler}) is  conservative in the sense that  each  total concentration $c_{ih} ^n(x)(i=1, \cdots,  m)$ remains unchanged in time,
\begin{align}
\label{conservediscrete}
 \sum_{j=1}^N \int_{I_j} c_{ih}^n  dx =   \sum_{j=1}^N \int_{I_j} c_{ih} ^{n+1} dx, \quad i=1,\cdots, m, \quad \quad t>0.
\end{align}

\item[2.] Assume that $c_{ih}^n(x)>0 $ in each $I_j$. There exists $\mu^*>0$ such that if the mesh ratio $\mu=\frac{\Delta t}{\Delta x^2} \in (0, \mu^*)$, then the fully discrete free energy
$$
F^n =\sum_{j=1}^N  \int_{I_j}  \left[ \sum_{i=1}^m c_{ih} ^n{\rm log} c_{ih}^n+\frac{1}{2} \left( \sum_{i=1}^m q_i c_{ih}^n+\rho_0 \right) \psi_h^n \right] dx +\frac{1}{2}\int_{\partial \Omega}\sigma  \psi_h^n ds
$$
satisfies
 \begin{align}
D_tF^n&\leq  - \frac{1}{2} \sum_{i=1}^m A_{c_{ih}^n}(p_{ih}^n, p_{ih}^n).
\label{fdEuler}
\end{align}
Moreover,
 \begin{align}\label{ff}
F^{n+1}\leq F^n,
\end{align}
provided that $\beta_0$ is suitably large,  and $\beta_1=0$ in $Fl(\psi_h)$  defined in (\ref{flux}) and (\ref{bpsi}).
\end{itemize}
\end{theorem}
\begin{proof} Set
$$
\rho^n_h=\sum_{i=1}^m q_i c_{ih}^n +\rho_0,
$$
and
sum over $j$ in (\ref{dgEuler}) to obtain
\begin{subequations}\label{dgEuler+}
\begin{align}
&  \int_{\Omega} D_t c_{ih}^n v_i dx =- A_{c_{ih}^n}(p_{ih}^n, v_{i}),\\
 & \int_{\Omega}p_{ih}^n r_i dx =\int_{\Omega}(q_i \psi_h^n + \log c_{ih}^n )r_{i}dx,\\
 & A_1(\psi_{h}^n, \eta)
 = 
 \int_{\Omega} \rho^n_h \eta dx + \int_{\partial \Omega} \sigma \eta ds.
 \end{align}
\end{subequations}
1.   Taking  $v_i=1$ in (\ref{dgEuler+}a) yields  \eqref{conservediscrete}.\\
2. Taking  $v_i=p_{ih}^n$ in (\ref{dgEuler+}a), and $r_i=D_t c^n_{ih}$ in (\ref{dgEuler+}b), respectively, and summing over $m$,  we obtain
$$
\int_{\Omega} \sum_{i=1}^m (D_t c^n_{ih}) p_{ih}^ndx = -\sum_{i=1}^m A_{c_{ih}^n}(p_{ih}^n, p_{ih}^n)
$$
and
$$
\int_{\Omega} \sum_{i=1}^m p_{ih}^n(D_t c^n_{ih}) dx =\int_{\Omega}\sum_{i=1}^m(q_i \psi_h^n + \log c_{ih}^n )D_t c^n_{ih}dx.
$$
From  (\ref{dgEuler+}c) we see that
\begin{align}\label{ppt}
A_1( \psi_h^{n+1}-\psi_h^n, \eta) =\int_{\Omega} ( \rho_h^{n+1} -\rho^n_h)\eta dx.
\end{align}
This with  $\eta=\psi_h^n $ and $ \rho_h^n=\sum_{i=1}^{m} q_i  c_{ih}^n +\rho_0(x)$
 yields
\begin{equation}\label{a1-}
 \int_{\Omega}  \sum_{i=1}^m q_i \psi_h^n(c^{n+1}_{ih}- c^{n}_{ih})dx= \int_{\Omega}(\rho_h^{n+1}-\rho_h^n)\psi_h^n dx=A_1(\psi_{h}^{n+1} -\psi_h^n, \psi_h^n).
\end{equation}
These relations lead to
\begin{equation}\label{aa}
\int_{\Omega}\sum_{i=1}^m \log c_{ih}^n (c^{n+1}_{ih} -c^{n}_{ih}) dx  = - \Delta t \sum_{i=1}^m A_{c_{ih}^n}(p_{ih}^n- p_{ih}^n) -A_1(\psi_{h}^{n+1} -\psi_h^n, \psi_h^n).
\end{equation}
Note that  (\ref{dgEuler+}c) with $\eta=\psi_h^{n+1}-\psi_h^n$ subtracted  by (\ref{ppt}) with $\eta=\psi_h^n$ gives
\begin{align*}
&    \int_{\Omega}  [\rho_h^n(\psi_h^{n+1}-\psi_h^n)  - ( \rho_{h}^{n+1}-\rho^n) \psi_h^n ]dx + \int_{\partial \Omega}\sigma (
   \psi_h^{n+1}-\psi_h^n) ds \\
 &  \quad =A_1(\psi_h^n, \psi_h^{n+1}-\psi_h^n )- A_1(\psi_h^{n+1}-\psi_h^n, \psi_h^n)=0.
\end{align*}
Also taking  $\eta=  \psi_{h}^{n+1}+\psi_h^n$ in (\ref{ppt}) we obtain
$$
 \int_{\Omega}(\rho_h^{n+1}-\rho_h^n)(\psi_h^{n+1}+\psi_h^n)dx=A_1(\psi_{h}^{n+1}-\psi_h^n,  \psi_h^{n+1}+\psi_h^n).
$$
Adding these up leads to
\begin{equation}\label{bb}
\int_{\partial \Omega}\sigma  (\psi_h^{n+1} -\psi_h^n)ds +
\int_{\Omega} ( \rho_h^{n+1} \psi_h^{n+1} -\rho_h^n \psi_h^n) dx=A_{1}(\psi_{h}^{n+1}-\psi_h^n,  \psi_h^{n+1}+\psi_h^n).
\end{equation}
With (\ref{aa}) and (\ref{bb}) we proceed to evaluate
\begin{align*}
F^{n+1}-F^n & = \sum_{i=1}^m \int_{\Omega} \left[  c_{ih} ^{n+1}{\rm log} c_{ih}^{n+1}- c_{ih} ^n{\rm log} c_{ih}^n\right] dx +\frac{1}{2}\int_{\partial \Omega}\sigma  (\psi_h^{n+1} -\psi_h^n)ds\\
& \qquad +
  \frac{1}{2} \int_{\Omega} ( \rho_h^{n+1} \psi_h^{n+1} -\rho_h^n \psi_h^n) dx\\
 & =\sum_{i=1}^m\int_{\Omega} \log c_{ih}^n (c^{n+1}_{ih} -c^{n}_{ih})dx + \sum_{i=1}^m\int_{\Omega}c^{n+1}_{ih} \log \left(\frac{c_{ih}^{n+1}}{c_{ih}^{n}}\right)dx\\
& \qquad  + \frac{1}{2}A_{1}(\psi_{h}^{n+1}-\psi_h^n,  \psi_h^{n+1}+\psi_h^n) \\
& =-\Delta t \sum_{i=1}^m A_{c_{ih}^n}(p_{ih}^n, p_{ih}^n)
+ \sum_{i=1}^m\int_{\Omega}c^{n+1}_{ih} \log \left(\frac{c_{ih}^{n+1}}{c_{ih}^{n}}\right)dx +\frac{1}{2}A_{1}(\psi_{h}^{n+1}-\psi_h^n,  \psi_h^{n+1}-\psi_h^n)\\
& = - \Delta t \sum_{i=1}^m A_{c_{ih}^n}(p_{ih}^n, p_{ih}^n)+ \sum_{i=1}^m\int_{\Omega}c^{n}_{ih} \log \left(\frac{c_{ih}^{n+1}}{c_{ih}^{n}}\right)dx+G^n \\
& \leq  - \Delta t \sum_{i=1}^m A_{c_{ih}^n}(p_{ih}^n, p_{ih}^n) +G^n,
\end{align*}
where the non-positivity of the second term is based on $\log X \leq X-1$ and the conservation of  $\sum_{i=1}^m\int_{\Omega}c^{n}_{ih}dx$. Here
$$
G^n= \sum_{i=1}^m\int_{\Omega}(c^{n+1}_{ih}-c^{n}_{ih})\log \left(\frac{c_{ih}^{n+1}}{c_{ih}^{n}}\right)dx +\frac{1}{2}A_{1}(\psi_{h}^{n+1}-\psi_h^n,  \psi_h^{n+1}-\psi_h^n)=: G_1^n +G_2^n,
$$
which is non-negative, yet small.

It remains to figure out a sufficient restriction on the mesh ratio $\mu=\Delta t/(\Delta x)^2$ so that
\begin{equation}\label{gn}
G^n \leq \frac{\Delta t}{2}\sum_{i=1}^m A_{c_{ih}^n}(p_{ih}^n, p_{ih}^n).
\end{equation}
In (\ref{dgEuler+}a),  we take $v_i=D_t c^n_{ih}$ and use the Young inequality $ab \leq \frac{1}{4\epsilon}a^2+\epsilon b^2$ to obtain
\begin{align*}
\sum_{j=1}^N\int_{I_j} v_i^2\, dx & = - \sum_{j=1}^N \int_{I_j} c_{ih}^n  \partial_xp_{ih}^n \partial_x v_i\,dx
	- \sum_{j=1}^{N-1} \{ c_{ih}^n\} \left.\left(\widehat{\partial_xp^n_{ih}} [v_i]+ \{ \partial_x v_i\}[p^n_{ih}]\right)
	\right|_{x_{j+\frac12}}\\
&\leq \frac{1}{4\epsilon_1h^2} \sum_{j=1}^N \int_{I_j} (c_{ih}^n)^2  |\partial_xp_{ih}^n|^2 \,dx +\epsilon_1 h^2 \sum_{j=1}^N \int_{I_j} |\partial_x v_i| ^2\,dx\\
   & +  \frac{1}{4\epsilon_2h} \sum_{j=1}^{N-1} \left.\{c_{ih}^n \}^2 |\widehat{\partial_xp^n_{ih}}|^2 \right|_{x_{j+\frac12}} +\epsilon_2h \sum_{j=1}^{N-1}\left. { [v_i]^2}\right|_{x_{j+\frac12}} \\
   & + \frac{1}{4\epsilon_3h^3} \sum_{j=1}^{N-1} \left.\{c_{ih}^n \}^2 [p^n_{ih}]^2\right|_{x_{j+\frac12}}  +\epsilon_3h^3 \sum_{j=1}^{N-1} \left.\{\partial_xv_i\}^2\right|_{x_{j+\frac12}} .
\end{align*}
The use of inequalities  (\ref{in}) in Lemma 3.2   leads to
\begin{align*}
& \epsilon_1  h^2 \sum_{j=1}^N \int_{I_j} |\partial_x v_i| ^2\,dx
    + \epsilon_2 h  \sum_{j=1}^{N-1} \left.{ [v_i]^2}\right|_{x_{j+\frac12}}
    + \epsilon_3h^3  \sum_{j=1}^{N-1}\left. [\partial_xv_i]^2 \right|_{x_{j+\frac12}} \\
  &    \leq(k+1)^2   ( k(k+2) \epsilon_1 + 4\epsilon_2+ k^3(k+2) \epsilon_3)
    \sum_{j=1}^N\int_{I_j} v_i^2\, dx \\
    & = \frac{3}{4}  \sum_{j=1}^N\int_{I_j} v_i^2\, dx,
 \end{align*}
if we choose $\epsilon_i$ as
 $$
 (4\epsilon_1)^{-1} =k(k+1)^2(k+2),  \; (4\epsilon_2)^{-1} =4 (k+1)^2, \quad (4\epsilon_3)^{-1} =k^3(k+1)^2(k+2).
 $$
 This gives
\begin{align}\label{inv}
& \frac{1}{4}\sum_{j=1}^N\int_{I_j} v_i^2\, dx \leq \frac{k(k+1)^2(k+2)}{h^2}\sum_{j=1}^N \int_{I_j} (c_{ih}^n)^2  |\partial_xp_{ih}^n|^2 \,dx \\ \notag
                                           & \qquad+\frac{k^3(k+1)^2(k+2)}{h^3} \sum_{j=1}^{N-1} \left.\{(c_{ih}^n) \}^2 [p^n_{ih}]^2\right|_{x_{j+\frac12}}+\frac{4(k+1)^2}{h}\sum_{j=1}^{N-1} \left.\{c_{ih}^n\}^2 |\widehat{\partial_xp^n_{ih}}|^2 \right|_{x_{j+\frac12}}.
\end{align}
It is clear that the first two terms are bounded by $h^{-2} \|c_{ih}^n(\cdot)\|_\infty \|p_{ih}^n\|_{c_{ih}^n, E}^2$.  We next show that the last term is also bounded by $h^{-2}\|c_{ih}^n(\cdot)\|_\infty \|p_{ih}^n\|_{c_{ih}^n, E}^2$, up to constant multiplication factors. Note that
\begin{align*}
  |\widehat{\partial_xp^n_{ih}}|^2\big|_{x_{j+\frac12}} & =
 \left|\{\partial_xp_{ih}^n\}+\beta_0\frac{[p^n_{ih}]}{h}+\beta_1h[\partial_x^2p^n_{ih}]\right|^2\\
&\leq  2  \left(\beta_0^2\frac{[p^n_{ih}]^2}{h^2}+\left(\{\partial_xp^n_{ih}\}+\beta_1h[\partial_x^2p^n_{ih}] \right)^2 \right).
\end{align*}
From (\ref{g1m}) it follows that
$$
 \left.\left(\{\partial_xp^n_{ih}\}+\beta_1h[\partial_x^2p^n_{ih}]\right)^2\right|_{x_{j+\frac12}}  \leq \frac{ \Gamma(\beta_1, c_{ih}^n )}{2h}
\left(\int_{I_j}+\int_{I_{j+1}}\right)c_{ih}^n |\partial_xp_{ih}|^2dx.
$$
Hence
\begin{align*}
 \sum_{j=1}^{N-1} \left. \{c_{ih}^n \}^2 |\widehat{\partial_xp^n_{ih}}|^2\right|_{x_{j+\frac12}}
 &
 \leq \frac{2}{h} \max\{\beta_0,  \Gamma(\beta_1, c_{ih}^n ) \{c_{ih}^n \} \}
 \|c_{ih}^n(\cdot)\|_\infty \|p^n_{ih}\|_{c_{ih}^n, E}^2\\
 & =\frac{2\beta_0}{h}\|c_{ih}^n(\cdot)\|_\infty \|p^n_{ih}\|_{c_{ih}^n, E}^2.
\end{align*}
Upon insertion into (\ref{inv}) we obtain
\begin{align*}
\sum_{j=1}^N\int_{I_j} v_i^2\, dx \leq  \frac{C(k, \beta_0) ||c_{ih}^n(\cdot)||_\infty}{h^2}\|p^n_{ih}\|_{c_{ih}^n, E}^2,
\end{align*}
where
\begin{equation}\label{ck}
C(k, \beta_0):= 4(k+1)^2 \left(k(k+2)\max\{1, k^2/\beta_0\}+ 8\beta_0\} \right).
\end{equation}
Note that
\begin{align*}
G_1^n & \leq \sum_{i=1}^m\int_{\Omega}\frac{(c_{ih}^{n+1}-c_{ih}^n)^2}{c_{ih}^n}dx=\sum_{i=1}^m\int_{\Omega}\frac{
v_i^2}{c_{ih}^n}dx(\Delta t)^2\\
& \leq \mu \Delta t \frac{C(k, \beta_0) ||c_{ih}^n(\cdot)||_\infty}{\min_{i, x}{c_{ih}^n}}\sum_{i=1}^m\|p^n_{ih}\|_{c_{ih}^n, E}^2\\
& \leq \mu \Delta t \frac{C(k, \beta_0) ||c_{ih}^n(\cdot)||_\infty}{\gamma \min_{i, x}{c_{ih}^n}}
\sum_{i=1}^mA_{c_{ih}^n}(p_{ih}^n, p_{ih}^n) \\
& \leq \frac{\Delta t}{4}\sum_{i=1}^mA_{c_{ih}^n}(p_{ih}^n, p_{ih}^n),
\end{align*}
if the mesh ratio satisfies
$$
\mu  \leq \frac{\gamma \min_{i, x}{c_{ih}^n}}{4C(k, \beta_0) \max_i ||c_{ih}^n(\cdot)||_\infty}.
$$
It remains to bound $G_2^n$.  From (\ref{ppt}) it follows that for $\xi:= \psi_h^{n+1}-\psi_h^n$,
\begin{align}\label{ppt+}
A_1(\xi, \xi) =\sum_{i=1}^m q_i \int_{\Omega} ( c_{ih}^{n+1} -c_{ih}^n)\xi dx= \Delta t
\sum_{i=1}^m q_i \int_{\Omega} v_i \xi dx \leq \Delta t
\sum_{i=1}^m |q_i| \| v_i\| \|\xi\|.
\end{align}
Note that $A_1(\cdot, \cdot)$  is a symmetric bilinear operator (since $\beta_1=0$ in (\ref{flux}) for $\psi_h$), and also that
 $A_1(\xi, \xi)=0$ if and only if $\xi\equiv const$, therefore
$$
c=\inf_{\{\xi \in V_h, \xi \not={\rm const}\}} \frac{h^2A_1(\xi, \xi)}{\|\xi\|^2},
$$
is positive.  A simple rescaling suggests that $c$ is also independent of $h$.   Hence  $\|\xi\|^2 \leq c^{-1}h^2A_1(\xi, \xi)$, which when inserted into (\ref{ppt+}), set $C_1= \sum_{i=1}^mq_i^2/c^2$,   leads to
\begin{align*}
A_1(\xi, \xi) & \leq C_1 (\Delta t)^2 h^2 \sum_{i=1}^m \| v_i\|^2  \leq C_1C(k, \beta_0) \max_i ||c_{ih}^n(\cdot)||_\infty  (\Delta t)^2\sum_{i=1}^m \|p^n_{ih}\|_{c_{ih}^n, E}^2\\
& \leq C_1 (\Delta t)^2 \frac{C(k, \beta_0) \max_i ||c_{ih}^n(\cdot)||_\infty}{\gamma} \sum_{i=1}^mA_{c_{ih}^n}(p_{ih}^n, p_{ih}^n) \\
& \leq \frac{\Delta t}{4}\sum_{i=1}^mA_{c_{ih}^n}(p_{ih}^n, p_{ih}^n),
\end{align*}
as long as the time step also satisfies
$$
\Delta t  \leq \frac{\gamma }{4C_1C(k, \beta_0) \max_i ||c_{ih}^n(\cdot)||_\infty}.
$$
Collecting the above estimates on $G_i^n (i=1, 2)$ we obtain (\ref{gn}), if we take
$$
\mu^* = \frac{\gamma }{4C(k, \beta_0) \max_i ||c_{ih}^n(\cdot)||_\infty} \min\{C_1^{-1}h^{-2}, \min_{i, x}{c_{ih}^n}\}.
$$
This ends the proof.
\end{proof}

\begin{remark} In Theorem \ref{thmEuler}, $c_{ih}(x)$ is assumed to be positive in each cell $I_j$ to make $p_{ih}=q_i\psi_h+\log c_{ih}$ well defined. In numerical simulations, we enforce this by imposing a limiter defined in  \eqref{ureconstruct}, based on positive cell averages. As shown in \cite{LW15},  such a limiter does not destroy the order of accuracy. Moreover, positivity of cell averages for each $c_{ih}$ is achieved for $\beta_0>1$ and $\beta_1\in (1/8, 1/4)$ when the coupling potential $\psi$ is zero. For the general PNP system, we numerically identify parameter pairs $(\beta_0, \beta_1)$ so that cell averages remain positive, as confirmed in Example 1.
\end{remark}

{
In our numerical simulation with $k=1,2,3$, we use the second order explicit Runge-Kutta method (RK2, also called Heun's method) for time discretization to solve the ODE system  of the form $\dot a=\mathfrak{L}(\textbf{a,t})$:
\begin{align}\label{RK2}
{\textbf{a}^{(1)}} &= \textbf{a}^n + \Delta t \mathfrak{L}(\textbf{a}^n,t_n), \nonumber \\
{\textbf{a}^{*}} &= \textbf{a}^{(1)} + \Delta t \mathfrak{L}(\textbf{a}^{{1}},t_{n+1}),  \\
\textbf{a}^{n+1} &=   \frac{1}{2}\textbf{a}^n + \frac{1}{2}\textbf{a}^{*}.\nonumber 
\end{align}

\begin{corollary}\label{corollary1} Consider the RK2 time discretization\eqref{RK2}. Assume that  $c_{ih}^n(x)$ and the intermediate states $c_{ih}^{(1)}$ and $c_{ih}^{*}$ are all positive in each $I_j$. There exists $\mu^*>0$ such that if the mesh ratio $\mu=\frac{\Delta t}{\Delta x^2} \in (0, \mu^*)$, then the fully discrete free energy
$$
F (c^n_h,\psi^n_h) =\sum_{j=1}^N  \int_{I_j}  \left[ \sum_{i=1}^m c_{ih} ^n{\rm log} c_{ih}^n+\frac{1}{2} \left( \sum_{i=1}^m q_i c_{ih}^n+\rho_0 \right) \psi_h^n \right] dx +\frac{1}{2}\int_{\partial \Omega}\sigma  \psi_h^n ds
$$
satisfies
$$
F^{n+1}\leq F^n,
$$
provided that $\beta_0$ is suitably large,  and $\beta_1=0$ in $Fl(\psi_h)$  defined in (\ref{flux}) and (\ref{bpsi}).
\end{corollary}
\begin{proof} From (\ref{ff}) in Theorem \ref{thmEuler} it follows that
 \begin{equation}\label{fs}
 F^{*}:=F(c^*_h,\psi^*_h)\leq F^{(1)}:=F(c^{(1)}_h,\psi^{(1)}_h)\leq F^n:=F(c^n_h,\psi^n_h).
 \end{equation}
 Note from (\ref{dgEuler+}), $F^n$ can also be written as
 $$
 F^n=\sum_{j=1}^N  \int_{I_j}  \sum_{i=1}^m c_{ih}^n {\rm log} \left(c_{ih}^n\right)dx+ \frac{1}{2}  A_1\left(\psi_{h}^n,
\psi_{h}^n \right).
 $$
 Hence
$$
F^{n+1} =\sum_{j=1}^N  \int_{I_j}  \left[ \sum_{i=1}^m \left(\frac{1}{2}c_{ih}^n +\frac{1}{2}c_{ih}^{*}\right){\rm log} \left(\frac{1}{2}c_{ih}^n +\frac{1}{2}c_{ih}^{*}\right)\right]dx+ \frac{1}{2}  A_1\left(\frac{1}{2}\psi_{h}^n +\frac{1}{2}\psi_{h}^{*},\frac{1}{2}\psi_{h}^n +\frac{1}{2}\psi_{h}^{*}\right),
$$
where by  \eqref{RK2}, we have used
 $$
 c_{ih}^{n+1}=\frac{1}{2}c_{ih}^n +\frac{1}{2}c_{ih}^{*};\quad \psi_{h}^{n+1}=\frac{1}{2}\psi_{h}^n +\frac{1}{2}\psi_{h}^{*}.
 $$
Using the fact that $A_1(w, w)$ is a convex functional in $w$ in the sense that
$$
A_1(\theta u+(1-\theta)v, \theta u+(1-\theta)v) \leq \theta A_1(u, u) +(1-\theta)A_1(v, v), \forall \theta \in [0, 1],
$$
which may be verified by the following identity
$$
\theta A_1(u, u) +(1-\theta)A_1(v, v) -A_1(\theta u+(1-\theta)v, \theta u+(1-\theta)v)=\theta(1-\theta)A_1(u-v, u-v)
$$
and $A_1(u-v, u-v) \geq 0$ from (\ref{amv}).
Thus
\begin{equation}\label{psiineq}
A_1\left(\frac{1}{2}\psi_{h}^n +\frac{1}{2}\psi_{h}^{*},\frac{1}{2}\psi_{h}^n +\frac{1}{2}\psi_{h}^{*}\right) \leq\frac{1}{2} A_1(\psi^n, \psi^n)+\frac{1}{2}A_1(\psi^*,\psi^*).
\end{equation}
Since $c\log c$ is convex in $c$, we also have
\begin{equation}\label{clncineq}
\left(\frac{1}{2}c_{ih}^n +\frac{1}{2}c_{ih}^{*}\right){\rm log} \left(\frac{1}{2}c_{ih}^n +\frac{1}{2}c_{ih}^{*}\right)\leq \frac{1}{2}c_{ih}^n {\rm log} \left(c_{ih}^n \right)+ \frac{1}{2}c_{ih}^* {\rm log} \left(c_{ih}^* \right)
\end{equation}
Equations \eqref{psiineq} and \eqref{clncineq} imply that
 $$
 F^{n+1}\leq \frac{1}{2}F^n +\frac{1}{2}F^{*}.
 $$
 This together with (\ref{fs}) leads to $F^{n+1} \leq F^n$, as claimed.
\end{proof}

\begin{remark}
Corollary \ref{corollary1} suggests that the DG scheme (\ref{dg}) with the time evolution by the strong stability preserving (SSP) Runge-Kutta method \cite{SO88}  does not increase the free energy at each time step, as long as the time step is suitably small. Hence high order SSP Runge-Kutta time discretization can be used for high order DG simulations, e.g., $k\geq 4$.
\end{remark}

\begin{remark} The time step restriction $\Delta t \sim O(\Delta x)^2$  is obviously a drawback of the explicit time discretization.
Usually one would use implicit in time discretization for diffusion and explicit time discretization for the nonlinear drift term (called IMEX in the literature) so that the time step restriction could be relaxed.  Unfortunately, formulation (\ref{cqp}) does not support such a separation.
\end{remark}
}

\subsection{Preservation of steady states}
If we start with initial data $c_{ih}^0$, already at steady states, i.e., ${\rm \log} c_{ih}^0  +
 q_i \psi^0_h(x)=C_i$,  it follows from (\ref{dgEuler}b) that $p_{ih}^0=C_i$.  Furthermore,  (\ref{dgEuler}a) implies that $c_{ih}^1=c_{ih}^0\in V_h$, which when inserted into  (\ref{dgEuler}c) gives $\psi_h^1=\psi_h^0$ (up to an constant, fixed to $0$);  hence  $
 {\rm \log} c_{ih}^1  + q_i \psi^1_h(x)=C_i.$  By induction we have
$$
{\rm \log} c_{ih}^n  +
 q_i \psi_h^n(x)=C_i \quad \forall n\in \mathbb{N}.
$$
This says that the DG scheme (\ref{dgEuler}) preserves  steady states.    Moreover, we can show that in some cases  the numerical solution tends
asymptotically toward a steady state, independent of initial data.  More precisely, we have the following result.
\begin{theorem}  Let the assumptions in Theorem \ref{thmEuler} be met,  and $(c_{ih}^n, p_{ih}^n, \psi_h^n) $ be the numerical solution to the fully discrete DG scheme (\ref{dgEuler}),   then the limits of $(c_{ih}^n, p_{ih}^n, \psi_h^n)$ as  $n\to \infty$ satisfy
$$
p^*_{ih} = C_i, \quad {\rm \log} c_{ih}^n  +
 q_i \psi_h^n(x)  \in C_i + V_h^\bot,
$$
where $C_i$ are  some constants.
\end{theorem}
\begin{proof}
Since $F^n$ is non-increasing and bounded from below,  we have
$$
\lim_{n \to \infty} F^n=\inf\{F^n\}.
$$
Observe from (\ref{fdEuler}) that
$$
F^{n+1} -F^{n} \leq -\frac{\Delta t}{2}\sum_{i=1}^m A_{c_{ih}^n}(p_{ih}^n, p_{ih}^n) \leq 0.
$$
When passing to the limit $n\to \infty$ we have $\lim_{n\to \infty} \sum_{i=1}^m A_{c_{ih}^n}(p_{ih}^n, p_{ih}^n)=0$.  This and the coercivity of $A_{c_{ih}^n}(p_{ih}^n, p_{ih}^n)$ imply the limit of $p_{ih}^n$, denoted by  $p_{ih}^*$, must be constant in each computational cell and the whole domain.  These when inserted into (\ref{dgEuler}b) gives the
desired result.  The proof is complete.
\end{proof}

\section{Numerical Implementation}
\subsection{Computing $\psi_h$ } In order to compute a unique $\psi_h$, we fix $\psi(a)$ as being given, and define
\begin{subequations}\label{fluxpsi}
\begin{align}
& Fl(\psi_h)(a) =\beta_0\frac{(\psi_{h}^+ - \psi(a))}{h}+\frac{1}{2}(\sigma_a +\psi_{hx}^+), \quad \{\psi\}=(\psi_h^++\psi(a))/2,\\
&Fl(\psi_h)(b) =\sigma_b, \quad \{\psi\}=\psi_h^-.
\end{align}
\end{subequations}
We add (\ref{dg}c) over all $j=1\cdots N$ and use the modified boundary condition (\ref{fluxpsi})  to obtain the following
\begin{align}\label{ga}
A(\psi_h, v)=L(v), \quad \forall v \in V_h,
\end{align}
where
\begin{align*}
A(\psi_h, v ) & = A_1(\psi_h, v)
+\left[ \left( \beta_0 \frac{\psi_h^+}{h}   +  \frac{1}{2} \psi_{hx} \right) v^+  + \frac{1}{2} \psi_{h}^+ v_{x}^+ \right]_{x_{1/2}}
\end{align*}
and
$$
L(v)=\int_{a}^b \rho(x) v(x)dx 
+\left[\left( \beta_0 \frac{\psi(a)}{h}  -  \frac{1}{2} \sigma_a \right) v^+   + \frac{1}{2} \psi(a) v_{x}^+\right]_{x_{1/2}} +\sigma_bv_{N+1/2}^-.
$$

\begin{lemma} For $\beta_0\geq \max\{\Gamma(\beta_1, 1), k^2\}$, and $\rho$ given, 
 there exists a unique  $\psi_h$ to (\ref{dg}c) with (\ref{flux}) and  the boundary fluxes in (\ref{fluxpsi}).
\end{lemma}
\begin{proof}
For the same reason as mentioned earlier, it suffices to prove the uniqueness.  Let $v=\psi_{h1}-\psi_{h2}$ for two
different solutions $\psi_{hi}, i=1, 2$, then
$$
A(v, v)=0.
$$
Note that
\begin{align} \notag
A(v, v) & = A_1(v, v)  
+\frac{\beta_0}{h}(v_{1/2}^+)^2 + (vv_x)^+_{1/2} \\ \notag
&  \geq  \gamma \|v\|_E^2  + \frac{1}{2}\int_{I_1}v_x^2dx + \frac{\beta_0}{2h}(v_{1/2}^+)^2 -\frac{h}{2\beta_0}(v_{x}^+)^2_{1/2} \\ \label{co}
& \geq \gamma \|v\|_E^2 +\frac{\beta_0}{2h}(v_{1/2}^+)^2,
\end{align}
provided that $\beta_0$ is large enough so that
$$
\beta_0 \geq \sup_{v\in P^k(I_1)}\frac{h(v_x)^2_{1/2}}{\int_{I_1}v_x^2dx}=\sup_{\eta \in P^{k-1}([-1, 1])} \frac{2u^2(1)}{\int_{-1}^1 u^2(\xi)d\xi}=\Gamma(0, 1)=k^2.
$$
Hence every term on the right of (\ref{co}) must be zero, which yields $v\equiv 0$. Uniqueness thus follows.
\end{proof}
\begin{remark} The above result provides a guide for the choices of $(\beta_0, \beta_1)$ in numerically solving the Poisson equation.  We shall take $\beta_1=0$ in solving the Poisson equation so that to also ensure the entropy dissipation property (see Theorem \ref{thm3.3}), hence it suffices to take $\beta_0>k^2=\Gamma(0, 1)$.
\end{remark}
\subsection{Positivity-preserving limiter}
In the scheme formulation involving the projection of $p_i = q_i \psi + log c_i$,  concentrations $c_{ih}$
needs to be strictly positive at each time step, and we follow \cite{LW15} to enforce positivity through some accuracy-preserving limiter based on positive cell averages.

Let $w_h \in  P^k(I_j)$  be an approximation to a smooth function $w(x) \geq 0$, with cell averages $\bar{w}_j>\delta$
for $\delta$ being some small positive parameter or zero.  We then consider another polynomial in $P^k(I_j)$
so that
\begin{equation}
w_h^{\delta}(x)= \bar{w}_j+\frac{\bar{w}_j-\delta}{\bar{w}_j-\min_{I_j} w_h(x)} (w_h(x)-\bar{w}_j),\quad
 \text{ if } \min_{I_j} w_h(x)<\delta. \label{ureconstruct}
\end{equation}
This reconstruction maintains same cell averages and satisfies  $$\min_{I_j} w^\delta(x)\geq\delta.$$
\begin{lemma} [cf. \cite{LW15}] If $\bar{w}_j>\delta$,  then $w^{\delta}$ satisfies the estimate
$$|w^{\delta}(x)-w_h(x)|\leq C(k) \left( ||w_h(x)-w(x)||_\infty+ \delta\right),\quad \forall x\in I_j,$$
where $C(k)$ is a constant depending on $k$.  This says that the reconstructed $w^{\delta}(x)$ in \eqref{ureconstruct}
does not destroy the accuracy when $\delta<h^{k+1}$.
\end{lemma}

\subsection{Algorithm}
The algorithm can be summarized in following steps.
\begin{itemize}
\item[1.] (Initialization) Project $c_i^{\rm in}(x)$ onto $V_h$, as formulated in (\ref{proj}),  to obtain $c_{ih}^{0}(x)$.
\item [2.] (Reconstruction)  From $c_{ih}^n(x)$,  apply,  if necessary, the reconstruction \eqref{ureconstruct} to $c_{ih}^n$ to ensure that in each cell $c_{ih}^n >\delta$.
\item[3.] (Poisson solver)  Solve (\ref{dg}c) to obtain $\psi_h^n$ subject to the modified boundary fluxes (\ref{fluxpsi}).
\item[4.] (Projection)  Solve (\ref{dg}b) to obtain $p_{ih}^n$.
\item[5.] (Update)  Solve (\ref{dg}a) to update $c_{ih}^{n+1}$ with some Runge-Kutta (RK) ODE solver.
\item[6.] Repeat steps 2-5 until final time $T$.
\end{itemize}

 The implementation details  are deferred to  Appendix A.


\section{Numerical Examples}
In this section, we present  a selected set of examples in order to numerically validate our DDG scheme.  In \S 5.1, we construct an example with exact solution known, and examine the order of accuracy by numerical convergence tests, while we quantify $l_1$ errors defined by
$$
\|u_h-u_{ref}\|_{l_1}= \sum_{j=1}^N \int_{I_j}|u_h(x) - u_{ref}(x)|dx,
$$
with the integral on $I_j$ evaluated by a $4$-point Gaussian quadrature method and $u_{ref}$ being the exact solution. Long time simulation is also performed to illustrate how the positivity of cell averages propagates when using  proper choices of $(\beta_0, \beta_1)$.    \S 5.2 is devoted to demonstrate the mass conservation,  energy dissipation and preservation of the steady state. In \S 5.3 and \S 5.4, we apply the DDG scheme to a non-monovalent system and a reduced single species system.

\subsection{Cell average and convergence test}
In $\Omega=[0,1]$, we consider
\begin{align*}
  \partial_t c_1 & = \partial_x (\partial_x c_1+  q_1 c_1 \partial_x \psi) +f_1, \\
  \partial_t c_2 & = \partial_x (\partial_x c_2+   q_2 c_2 \partial_x \psi) +f_2, \\
- \partial_x^2 \psi & =  q_1c_1+ q_2c_2,\\
\partial_x \psi(t,0)& =0,  \quad \partial_x \psi(t, 1)=-e^{-t}/60, \\
{\partial_x c_i} &+q_ic_i {\partial_x \psi}=0,
 \quad x=0, 1,
\end{align*}
with
\begin{align*}
 f_1 &=\frac{(50x^9-198x^8+292x^7-189x^6+45x^5)}{30e^{2t}}+\frac{ (-x^4+2x^3-13x^2+12x-2)}{e^{t}},   \\
 f_2&=\frac{(x - 1)(110x^9 - 430x^8 + 623x^7 - 393x^6+90x^5)}{60e^{2t}} + \frac{(x-1)(x^4 - 2x^3 + 21x^2 - 16x + 2)}{e^{t}}.
\end{align*}
This system, with $q_1=1$ and $q_2=-1$,  admits exact solutions
\begin{align*}
  c_1 & =x^2(1-x)^2e^{-t}, \\
 c_2  &= x^2(1-x)^3e^{-t}, \\
\psi & =  -(10x^7-28x^6+21x^5)e^{-t}/420.
\end{align*}
We also  set $\psi(t,0)=0$ to pick out a particular solution since $\psi$ is unique up to an additive constant. This extra condition is numerically enforced according to \eqref{fluxpsi}.

We first test positivity of cell averages  for $P^2$ polynomials with $(\beta_0,\beta_1)=(4,1/12)$ in (\ref{flux}) for $p_{ih}$. Note that the reconstruction is necessary in this example since $c_i=0<\delta$ at $x=0, 1$. 
Our simulation with the reconstruction \eqref{ureconstruct} up to $T=100$ indicates that the cell averages remain positive.

Table \ref{tab:ex1LT} displays both $l_1$ errors and orders of convergence when using $P^{k}$ elements at $T=0.1$. We  observe that order of convergence is roughly of $k+1$. Figure \ref{fig:ex1} shows the numerical solution at different times. In the top of Figure \ref{fig:ex1}, we observe that the numerical solutions (dots) match the exact solutions (solid line) at $t=1$ very well.   At $t=5$,  our numerical approximation still captures the solution  profile very well when  the magnitude of the concentrations is close to zero ($\sim10^{-4}$). In the presence of  the source terms $f_1$ and $f_2$, one should not expect to have either mass conservation or free energy decay.

\begin{table}[!htb]
\caption{Error table of Example 1 at $T=0.1$}
\begin{tabular}{ |c|l|c|c| c|c| c|c| }
\hline
$(k,\beta_0, \beta_1)$& h & $c_1$ error & order & $c_2$ error & order & $\psi$ error & order \\ \hline
\multirow{4}{*}{$(1,2,-)$}
 & 0.2  &      0.023279 &   --       &     0.031295   &  --    &       0.0033241&--      \\
 & 0.1  &     0.0037603 &    2.5043  &     0.0059588   &     2.2582     &      0.0009351   &     1.8578\\
 & 0.05  &  0.00065589  &  2.4414   &    0.0012548   &     2.1909      &     0.0002603    &   1.8718\\
 & 0.025 &  0.00012745  & 2.3635   &   0.00028581    &   2.1343       &   6.9808e-05      &   1.8987   \\ \hline
 \multirow{4}{*}{$(2,4,{1}/{12})$}
 & 0.2  &    0.0028937  &   --              &      0.0030675   &  --    &   0.0012417  &--      \\
 & 0.1  &   0.00018926 &      3.6436    &   0.00024835   &    3.4352    &   0.00010034 &   3.4636  \\
 & 0.05  &  1.391e-05   &     3.4981  &    2.2705e-05   &      3.3395     & 9.1444e-06&      3.3808  \\
 & 0.025 & 1.4824e-06   &   3.2301   &   2.4238e-06    &     3.2277     &  9.2476e-07  &   3.3057   \\ \hline
 \multirow{4}{*}{$(3,15,{1}/{4})$}

  & 0.2  &    0.0030963 &   --                 &     0.0029231    &  --    &     0.0011002  &--      \\
 & 0.1  &   0.00023282 &    4.0764        &    0.00021924    &    3.9254   &    7.4195e-05   &   4.3897  \\
 & 0.05  & 1.7512e-05  &    4.2480         &     1.6857e-05  &    4.0196    &     5.4161e-06 &     4.6393  \\
 & 0.025 &6.4483e-07  &    4.7633       &      8.3344e-07  &   4.3381    &      1.1946e-07 &    5.5027  \\ \hline
\end{tabular}
\label{tab:ex1LT}
\end{table}

\begin{figure}[!htb]
\caption{Numerical solution versus exact solution at $t=1$ and $t=5$}
\centering
\begin{tabular}{cc}
\includegraphics[width=\textwidth]{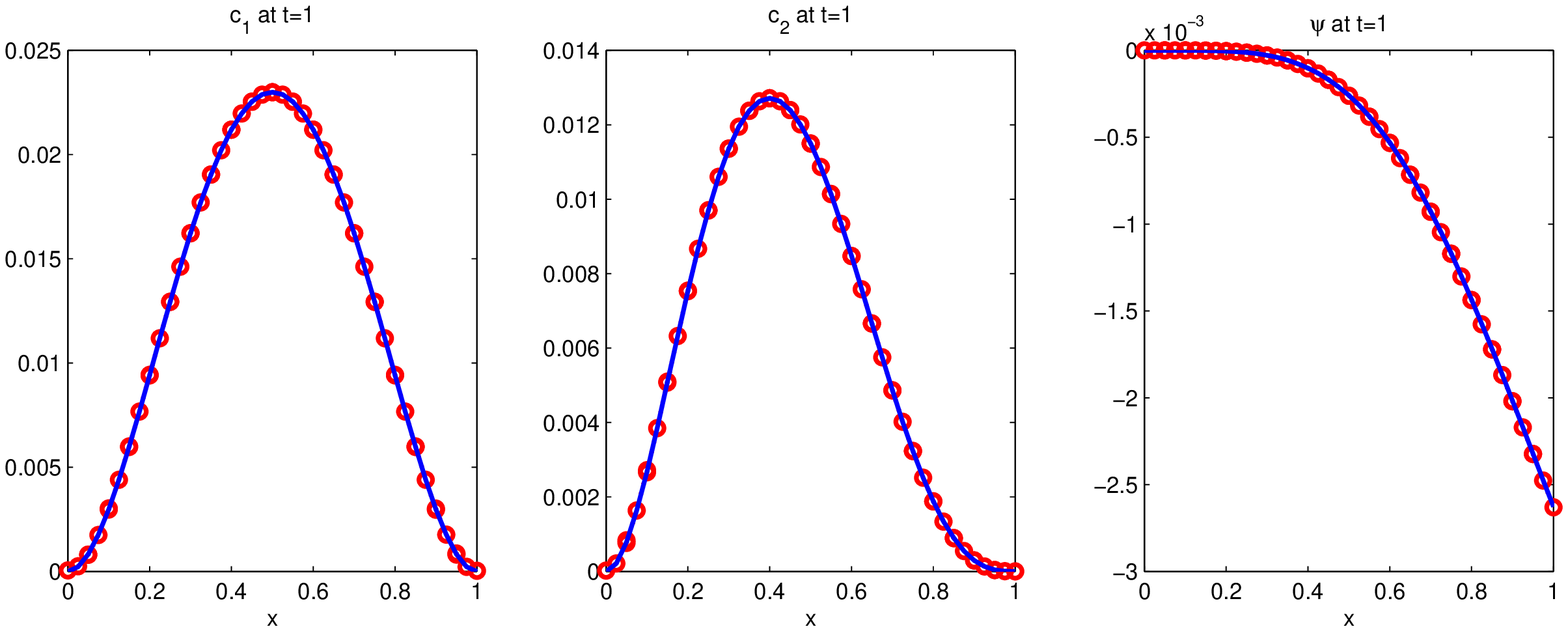}\\
\includegraphics[width=\textwidth]{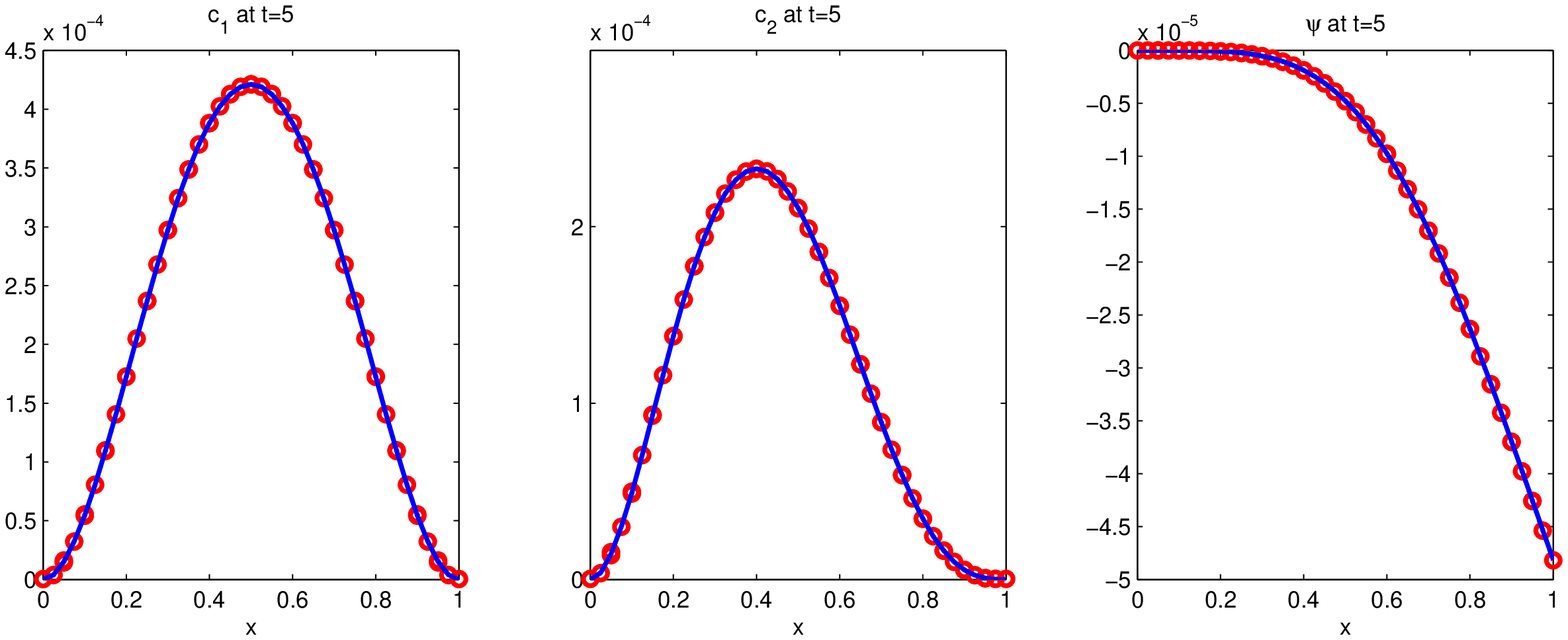} 
\end{tabular}
{Solid line: exact solutions; dots: numerical solutions}
\label{fig:ex1}
\end{figure}

\subsection{Mass conservation and free energy dissipation}
We consider the following problem on the domain $[0,1]$,
\begin{align*}\label{PNPtwo}
  \partial_t c_i &=\partial_x (\partial_x c_i+ q_ic_i\partial_x \psi), \quad i=1, 2 \\ 
- \partial_x^2\psi&= q_1c_1+q_2c_2, \\ 
   {\partial_x \psi}=0,\quad &{\partial_x c_i}=0,\quad i=1,2,\quad x=0, 1,
\end{align*}
where $q_1$ and $q_2$ are set to be $1$ and $-1$, respectively, with initial  conditions
\begin{align*}
 \quad c_1^{\rm in}(x)=1+\pi\sin(\pi x), \quad c_2^{\rm in}(x)= 4-2x ,
\end{align*}
which satisfies the compatibility condition \eqref{compatibility}.

With zero flux for concentration $c_1$ and $c_2$, this model example is to test the conservation of total mass and free energy decay. In Figure \ref{fig:ex2} (top), we see the snapshots of $c_1$, $c_2$ and $\psi$ at $t=0, 0.01, 0.1,0.8, 1$. We observe that the solutions at $t=0.8$ and $t=1$ are indistinguishable. Obviously the solution is converging to the steady state, which has constant $c_1=3$, $c_2=3$ and $\psi=0$. Figure \ref{fig:ex2} (bottom) shows the energy decay (see the change on the right vertical axis)  and conservation of mass (see the left vertical axis). We see that the total mass of $c_1$ and $c_2$ stays constant all the time while the free energy is decreasing monotonically. In fact the free energy levels off after $t=0.2$, at which the system is already in steady state. 

\begin{figure}[!htb]
\caption{Temporal evolution of the solutions}
\centering
\begin{tabular}{cc}
\includegraphics[width=1.05\textwidth]{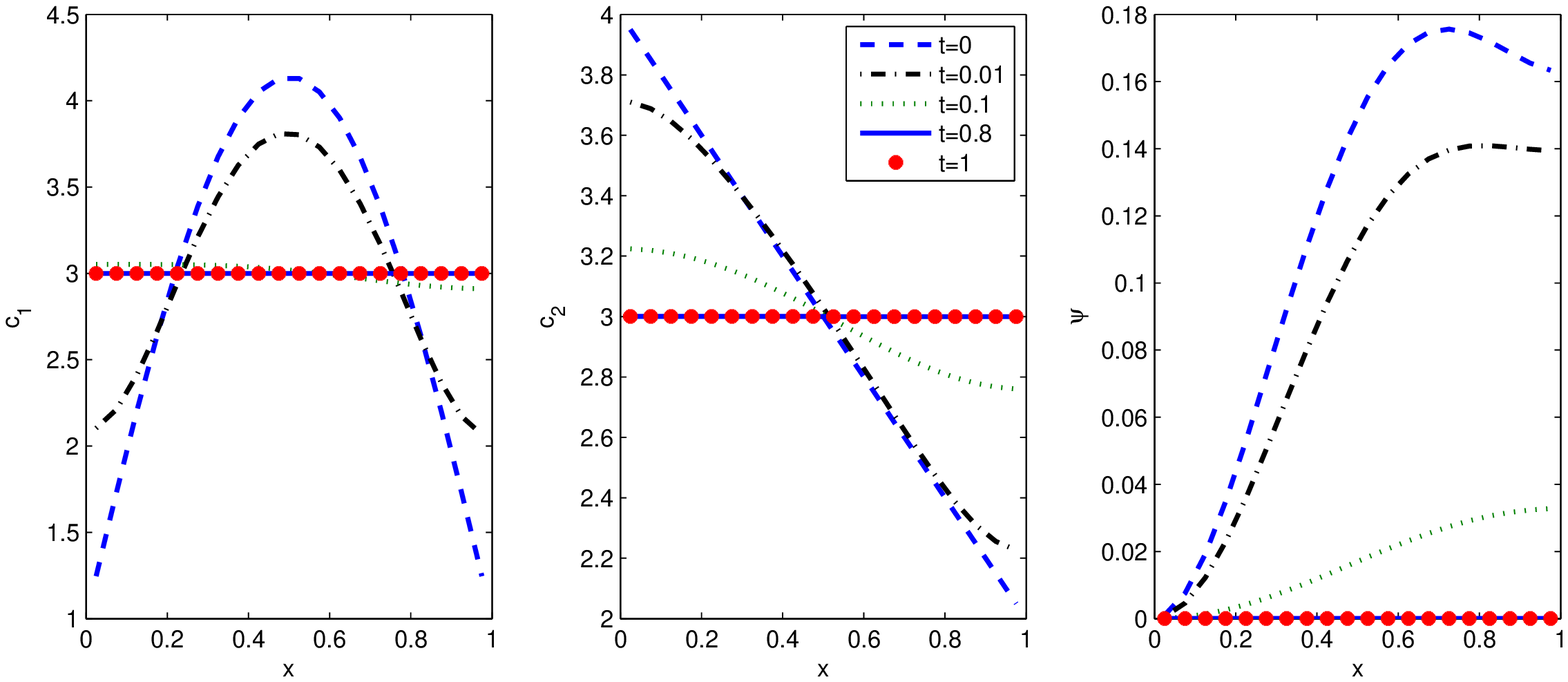}\\
\includegraphics[width=\textwidth]{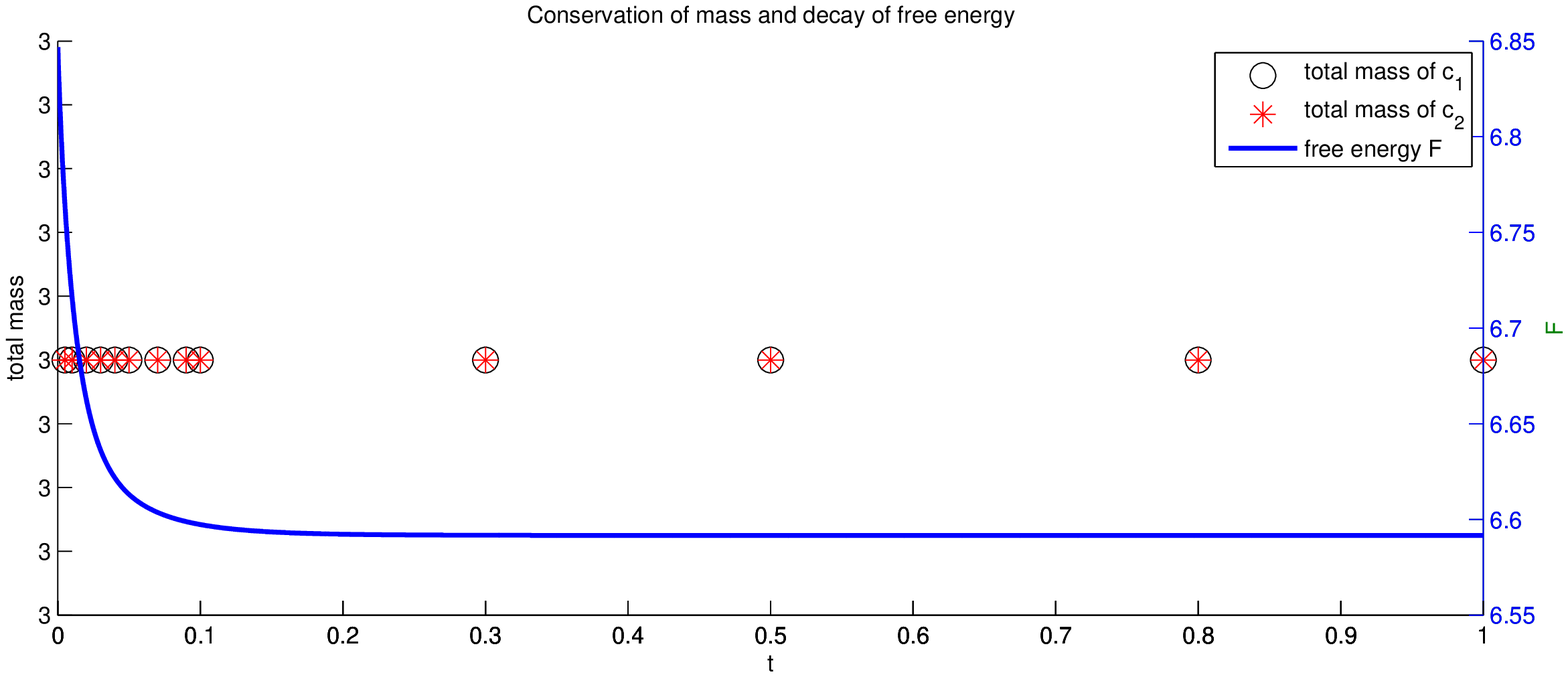} 
\end{tabular}
\label{fig:ex2}
\end{figure}

\subsection{Non-monovalent system and nonzero fixed charge}
We consider the following non-monovalent system (monovalent if $q_1=-q_2=1$) with nonzero fixed charge $\rho_0$ in $[0,1]$.
\begin{align*}
 \partial_t c_1&=\partial_x (\partial_x c_1 + q_1c_1\psi_x ),  \\
 \partial_t c_2&=\partial_x (\partial_x c_2 + q_2 c_2\psi_x ),   \\
 -\partial_x^2\psi &= q_1c_1+q_2c_2 +\rho_0,
\end{align*}
with $q_1=1$, $q_2=-2$ and $\rho_0=12(x-0.5)^2$. The initial and boundary conditions are
\begin{align*}
 \quad c_1^{\rm in}(x)&=2+12(x-0.5)^2, \quad c_2^{\rm in}(x)= 1+2x, \\
& {\partial_x c_i} +q_ic_i {\partial_x \psi} =0,\quad x=0, 1,\\
& \partial_x\psi(t, 0)=\partial_x \psi(t, 1)=0,
\end{align*}
where the compatibility condition \eqref{compatibility} is satisfied since $\int_0^1(q_1c_1^{\rm in} +q_2c_2^{\rm in}+\rho_0)dx=0$.

In Figure \ref{fig:ex3} (top) are the snapshots of $c_1$, $c_2$ and $\psi$ at $t=0, 0.01, 0.1, 0.8, 1$.  Figure \ref{fig:ex3} (bottom) shows the energy decay (see the change on the right vertical axis)  and conservation of mass (see the left vertical axis). The concentrations $c_1$ and $c_2$ have different total mass but both are conserved in time.  We observe that the system is at steady states after $t=0.2$, which are no longer constants due to the nonzero fixed charge $\rho_0$.

\begin{figure}[!htb]
\caption{Non-monovalent system and nonzero fixed charge}
\centering
\begin{tabular}{cc}
\includegraphics[width=1.05\textwidth]{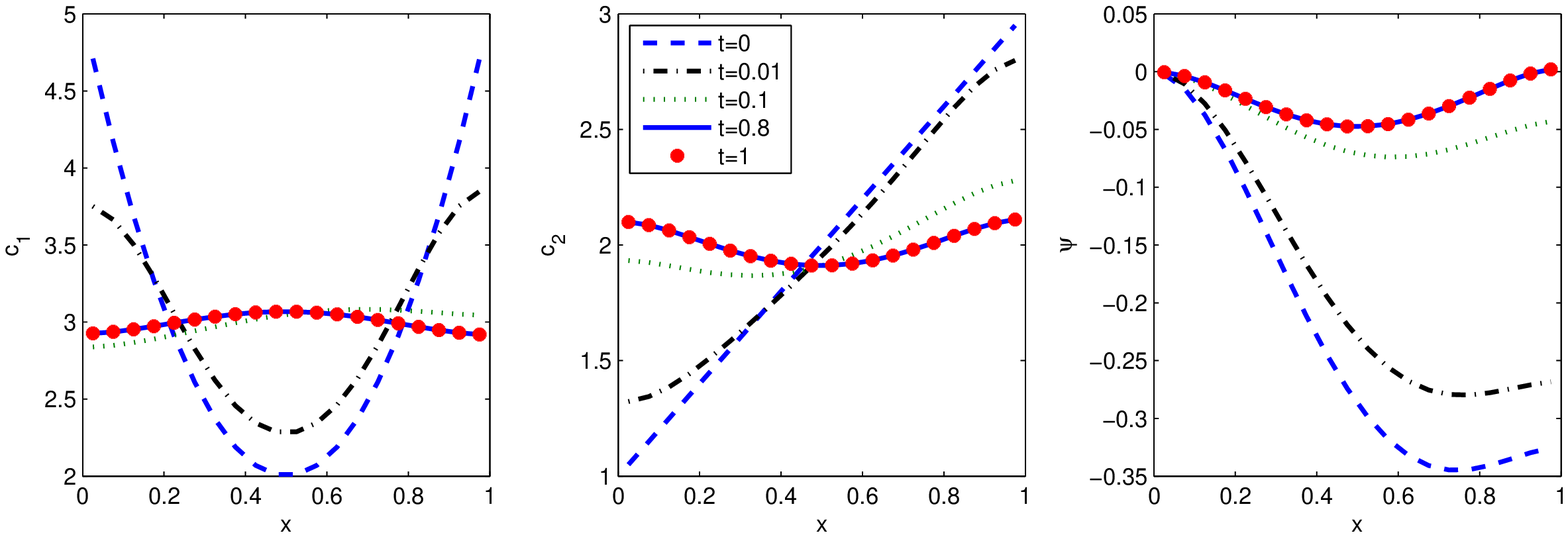}\\
\includegraphics[width=\textwidth]{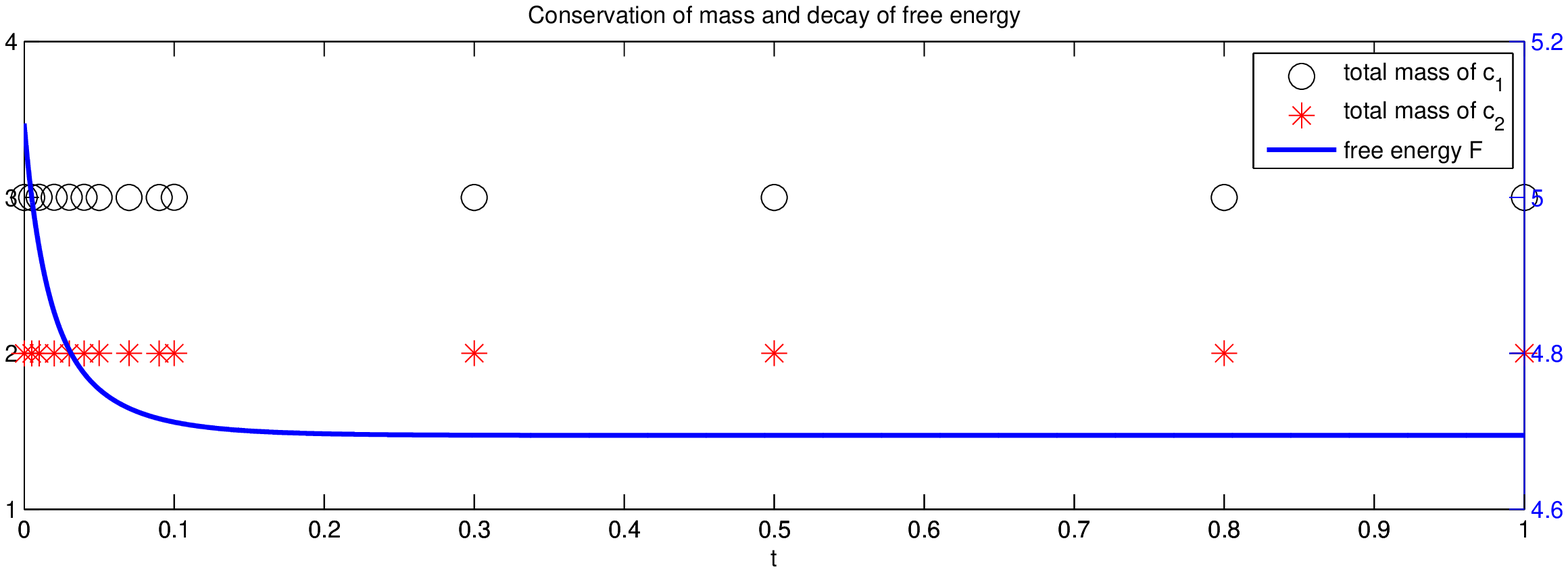} 
\end{tabular}
\label{fig:ex3}
\end{figure}
\subsection{Single species} Finally we consider the reduced model \eqref{PNP1} of  single species.
The problem is
\begin{align*}
 \partial_t c & = \partial_x(\partial_x c +  c  \partial_x \psi),  \quad x\in [0, 1], \;t>0, \\
-  \partial_x^2 \psi & =  c ,  \quad x\in [0, 1],
\end{align*}
subject to initial and boundary conditions
$$
c(0, x)=2-x; \quad \partial_x\psi(t,0) =0;  \quad \partial_x\psi(t, 1)=-3/2; \quad \partial_xc+c\partial_x\psi=0,\quad x=0,1,
$$
which satisfies the compatibility condition \eqref{compatibility}.

Figure \ref{fig:ex4}  displays the dynamic behavior of $c$ (top left) and of $\psi$ (top right), as well as the mass conservation and energy decay (bottom). Note that the steady state of the concentration $c$ is not a constant in this case due to the nonzero boundary flux $\partial_x\psi(t, 1)$, i.e., ions are attracted to the boundary $x=1$.

\begin{figure}[!htb]
\caption{Single species}
\centering
\begin{tabular}{cc}
\includegraphics[width=0.8\textwidth]{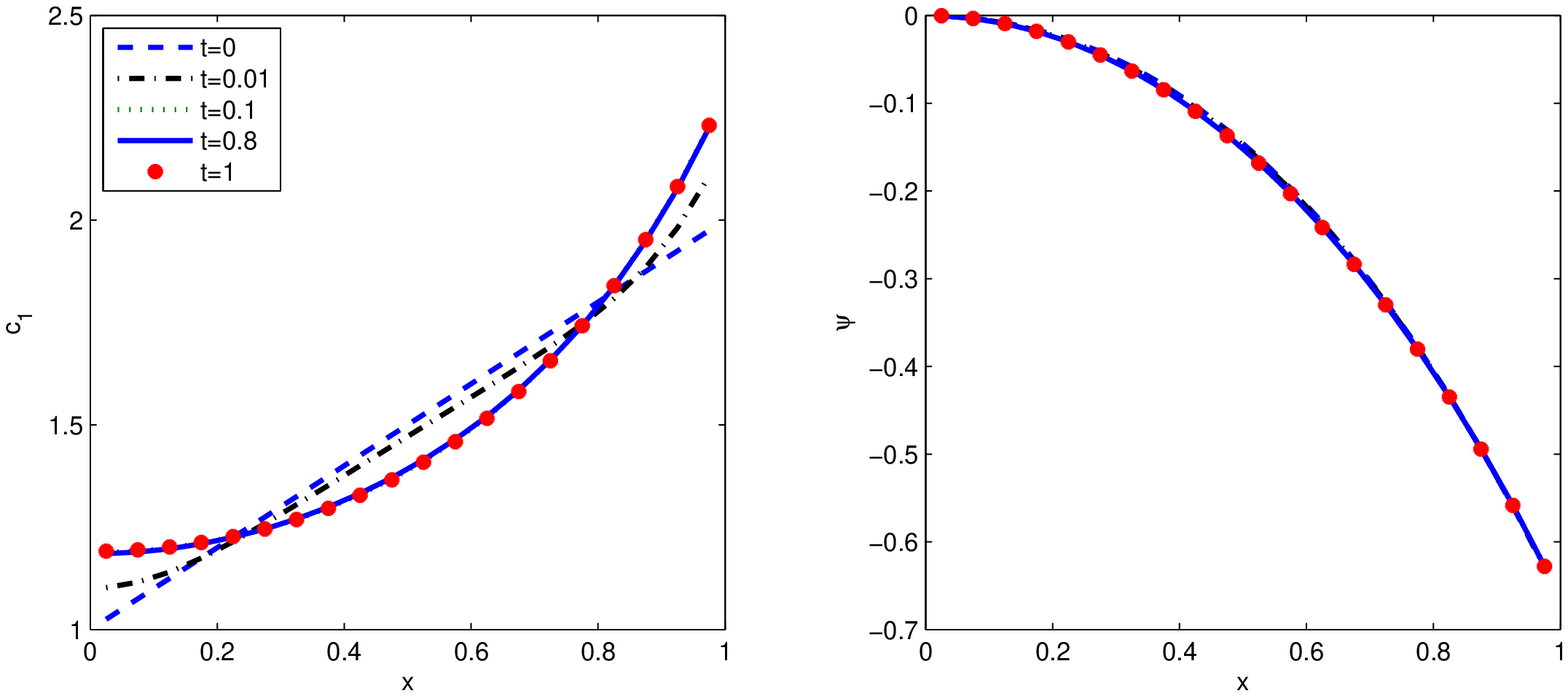}\\
\includegraphics[width=0.8\textwidth]{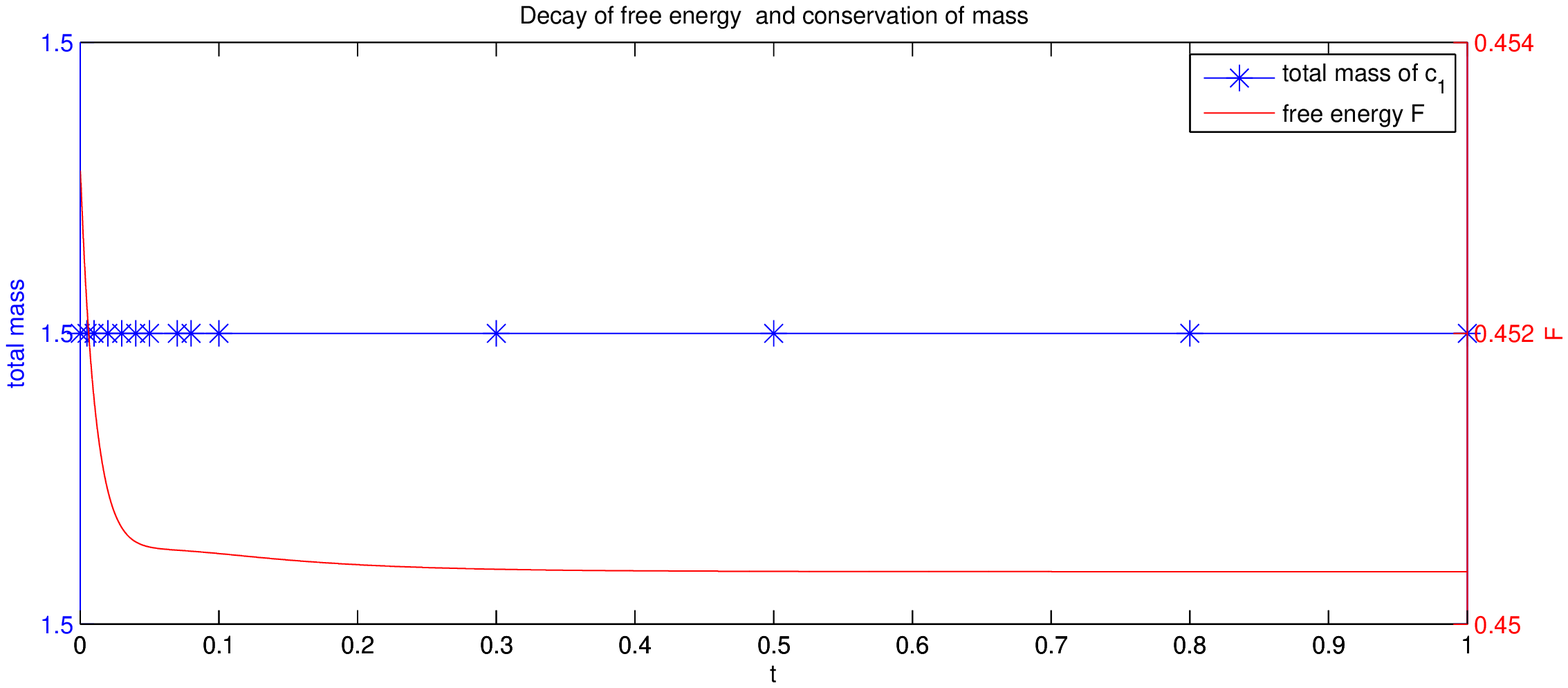} 
\end{tabular}
\label{fig:ex4}
\end{figure}

\section{Concluding remarks}
In this paper, we developed an arbitrary high order DG method to solve initial boundary value problems for the Poisson-Nernst-Planck system,  which is a mean field type model for concentrations of chemical species. The semi-discrete DG method has been shown to conserve the mass and satisfy the corresponding discrete  free energy dissipation law.  The fully discrete DG method with the Euler forward time discretization also satisfies the free energy dissipation law under a  restriction on the time step relative to the square of the spatial mesh size. Moreover, the free energy does not increase at each time step for the RK2 method (\ref{RK2}) as well as for the strong stability preserving (SSP) Runge-Kutta method of any high order. The method also preserves the steady states. These nice properties are also confirmed by our numerical tests. For proper choices of numerical flux parameters $(\beta_0, \beta_1)$ in our numerical experiments, we find that each cell average remains positive if it is initially positive. This is needed to update numerical solutions at each time step so they are positive.  In future work we will apply our method to multi-dimensional problems and explore  more efficient ways to preserve positivity of numerical solutions.

\section*{Acknowledgments}  This work was supported by the National Science Foundation under Grant DMS1312636 and by NSF Grant RNMS (Ki-Net) 1107291.

\appendix
  \renewcommand{\theequation}{A-\arabic{equation}}
  \setcounter{equation}{0}  
  \section{} 

In this appendix, we give details related to the implementation of the method.

The $k$th order basis functions in a 1-D standard reference element $\xi\in [-1, 1]$ are taken as
the Legendre polynomials $\{L_l(\xi)\}_{l=0}^k$. The numerical solutions $c_{ij}(t, x)$, $p_{ij}(t, x)$ and $\psi_{i}(t, x)$ in each cell $I_j$, after dropping subscript $h$ for convenience,  can be expressed as
\begin{align*}
 c_{ij}(t, x) &=\sum_{l=0}^kc_{ij}^l(t)L_l(\xi)=L^\top(\xi)c_{ij}(t), \\
 \psi_j(t, x) &=\sum_{l=0}^k\psi_j^l(t)L_l(\xi)=L^\top(\xi)\psi_{j}(t), \\
p_{ij}(t, x) & =\sum_{l=0}^kp_{ij}^l(t)L_l(\xi)=L^\top(\xi)p_{ij}(t),
\end{align*}
using the map $x=x_j+\frac{h}{2}\xi$, with notation $L^\top=(L_0, L_1, \cdots, L_k)$ and $c_{ij}=(c_{ij}^0, \cdots, c_{ij}^k)^\top$.  With such an expression, the numerical fluxes at $x_{j+1/2}$ become
\begin{align*} \notag
h \widehat{\partial_xp_{i}} & =\beta_0 (L^\top(-1)p_{i,j+1}-L^\top(1) p_{ij}) +(L^\top_\xi(-1)p_{i,j+1}+L^\top_\xi(1)p_{ij}) \\
& \qquad +4\beta_1(L^\top_{\xi\xi}(-1) p_{i,j+1} -L^\top_{\xi\xi} (1) p_{ij}),\\
h \widehat{\partial_x\psi} & =\beta_0 (L^\top(-1)\psi_{j+1}-L^\top(1) \psi_{j})+L^\top_\xi(-1) \psi_{j+1} +L^\top_\xi(1)\psi_j,\\
\widehat \psi & = \frac{1}{2} (L^\top(1) \psi_j+L^\top(-1) \psi_{j+1}).
\end{align*}
This when substituted into (\ref{dg}c) gives
\begin{equation}\label{ma}
A\psi_{j-1} +B\psi_j +C\psi_{j+1} = hK \
  \sum_{i=1}^m q_i c_{ij} +  \frac{h^2}{2} \sum_{n=1}^{Q_1} \omega_n \rho_0(x_j+hs_n/2) L(s_n),
, \quad 2\leq j\leq N-1,
\end{equation}
where
\begin{align*}
A =& -L(-1)(\beta_0L(1)-L_\xi(1))^\top-L_\xi(-1)L^\top(1),\\
B =& 2 \int_{-1}^1 L_\xi L_\xi^\top d\xi +L(-1)(\beta_0L(-1)+L_\xi(-1))^\top {+} L(1)(\beta_0L(1)-L_\xi(1))^\top\\
      &+L_\xi(-1)L^\top(-1)-L_\xi(1)L^\top(1),\\
C =&   -L(1)(\beta_0L(-1)+L_\xi(-1))^\top+L_\xi(1)L^\top(-1),\\
K  =& \frac{h}{2}\int_{-1}^1 L(\xi)L^\top(\xi)d\xi.
\end{align*}
Here $Q_1$-point Gauss quadrature rule on the interval $(-1, 1)$ is used to integrate $\rho_0(x_j+h\xi/2)L(\xi)$. We choose $Q_1\geq \frac{k+2}{2}$ points so that the quadrature rule has accuracy of at least $O(h^{k+2})$ order. Boundary conditions are specified in section \ref{BC}. Note that the resulting  $(k+1)N\times (k+1)N$  matrix  is a sparse block matrix, which  can be inverted efficiently.

We  then compute (\ref{dg}b) by the same quadrature rule
\begin{equation}
K p_{ij}=q_iK \psi_j + \frac{h}{2} \sum_{n=1}^{Q_1} \omega_n log(L^\top(s_n)c_{ij})L(s_n).\label{getq}
\end{equation}

Finally we simplify  (\ref{dg}a) to get the following ODE system
\begin{equation}\label{3}
K \dot c_{ij}=\frac{2}{h} R_1+\frac{1}{2h}( R_2+R_3),  \quad 2\leq j \leq N-1,
\end{equation}
where
\begin{align*}
R_1 & =-\sum_{n=1}^{Q_2}  \omega_n L^\top(s_n)c_{ij} L_\xi^\top(s_n)p_{ij} L_\xi(s_n),\\
R_2 & =(L^\top(1)c_{ij}+L^\top(-1)c_{i,j+1})( - D^\top  p_{ij}+ E^\top p_{i,j+1})L(1) \\
& \qquad - (L^\top(1)c_{i,j-1}+L^\top(-1)c_{ij})( - D^\top p_{i,j-1}+ E^\top p_{ij})L(-1)=R_2^+-R_2^-, \\
R_3 & =(L^\top(1)c_{ij}+L^\top(-1)c_{i,j+1})( L^\top(1) p_{ij}- L^\top(-1)p_{i,j+1})L_\xi(1)\\
&\qquad   {+} (L^\top(1)c_{i,j-1}+L^\top(-1)c_{ij})(L^\top(1) p_{i,j-1}- L^\top({-1})p_{ij})L_\xi(-1)=: R_3^+ {+} R_3^-.
\end{align*}
Here
$$
D=\beta_0L(1)-L_\xi(1)+4\beta_1L_{\xi\xi}(1), \quad  E =\beta_0L(-1)+L_\xi(-1)+4\beta_1L_{\xi\xi}(-1).
$$
In the evaluation of $R_1$,  we choose $Q_2 \geq \frac{k+4}{2}$ points. In implementation, since $Q_2>Q_1$, we simply use $Q_2$ Gaussian quadrature points in integrating both $\rho_0$ and $\log(c_i)$.

At two end cells, $R_1$ is still valid. The zero flux conditions (\ref{PNP}d) are used to obtain $R_2$ and $R_3$, i.e., at $j=1$, we use  $R_2=R_2^+, R_3=R_3^+$, and at $j=N$ we use  $R_2=-R_2^-, R_3=-R_3^-$.
The discrete ODE system is then solved by the RK2 method (\ref{RK2}).

\end{document}